\documentclass[a4paper,reqno]{amsart}
\usepackage{dsfont}
\usepackage{newcent} 
\usepackage[T1]{fontenc}
\usepackage{amsmath,amssymb,amsfonts,amsthm}
\usepackage{mathtools}
\usepackage{newlfont}
\usepackage{fancyhdr,fancyvrb}
\usepackage{graphicx}
\usepackage{nextpage}
\usepackage[dvipsnames,usenames]{color}
\usepackage[colorlinks=true, linkcolor=blue, citecolor=ForestGreen]{hyperref}
\usepackage{tikz}
\usepackage{tikz-3dplot}
\usepackage{mathrsfs}

\newcommand{\mygraphic}[1]{\includegraphics[height=#1]{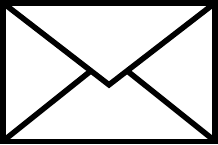}}
\newcommand{\myenv}{(\raisebox{0pt}{\mygraphic{.6em}})}

\newcommand{\R}[1]{\mathbb{R}^{#1}}
\renewcommand{\S}[1]{\mathbb{S}^{#1}}
\newcommand{\T}{\mathbb{T}}

\newcommand{\ba}{\mathbf{a}}
\newcommand{\bb}{\mathbf{b}}
	
\newcommand{\be}{\mathbf{e}}
\newcommand{\bef}{\mathbf{f}}
\newcommand{\bF}{\mathbf{F}}
\newcommand{\bM}{\mathbf{M}}
\newcommand{\bN}{\mathbf{N}}
\newcommand{\bT}{\mathbf{T}}

\newcommand{\bv}{\mathbf{v}}
\newcommand{\bx}{\mathbf{x}}

\newcommand{\bz}{\mathbf{z}}

\newcommand{\btau}{\boldsymbol{\tau}}

\newcommand{\bomega}{\boldsymbol{\omega}}

\newcommand{\fLie}{\mathfrak{Lie}}
\newcommand{\fse}{\mathfrak{se}}

\newcommand{\cL}{\mathcal L}
\newcommand{\cM}{\mathcal M}

\newcommand{\cP}{\mathcal P}

\newcommand{\cS}{\mathcal S}

\newcommand{\cZ}{\mathcal Z}

\newcommand{\scrL}{\mathscr{L}}

\newcommand{\fg}{\mathfrak{g}}

\newcommand{\de}{\mathrm d}

\newcommand{\sh}{\mathrm{sh}}

\newcommand{\scp}[2]{\left\langle#1,#2\right\rangle}

\renewcommand{\geq}{\geqslant}

\newcommand{\average}{{\mathchoice {\kern1ex\vcenter{\hrule
height.4pt width 8pt depth0pt}
\kern-11pt} {\kern1ex\vcenter{\hrule height.4pt width 4.3pt
depth0pt} \kern-7pt} {} {} }}

\renewcommand{\vec}[2]{\left(\begin{array}{c}
                    #1 \\
                    #2
                    \end{array}\right)}
\newcommand{\mat}[4]{\left[\begin{array}{cc}
                    #1 & #2 \\
                    #3 & #4
                    \end{array}\right]}

\mathchardef\emptyset="001F

\numberwithin{equation}{section}

\newtheorem{defin}{Definition}[section]
\newtheorem{remark}[defin]{Remark}
\newtheorem{theorem}[defin]{Theorem}

\newtheorem{proposition}[defin]{Proposition}

\title[Controllability of the 3D $N$-link swimmer]{The $N$-link swimmer in three dimensions: \\controllability and optimality results}%
\author{Roberto Marchello}
\address[R.~Marchello]{Dipartimento di Scienze Matematiche ``G.~L.~Lagrange'', Politecnico di Torino, Corso Duca degli Abruzzi, 24, 10129 Torino, Italy}
\email{s277438@studenti.polito.it}
\author{Marco Morandotti}
\address[M.~Morandotti]{Dipartimento di Scienze Matematiche ``G.~L.~Lagrange'', Politecnico di Torino, Corso Duca degli Abruzzi, 24, 10129 Torino, Italy}
\email{marco.morandotti@polito.it}
\author{Henry Shum}
\address[H.~Shum]{Department of Applied Mathematics, University of Waterloo, 200 University Avenue West, Waterloo, ON, Canada N2L 3G1}
\email{henry.shum@uwaterloo.ca}
\author{Marta Zoppello}
\address[M.~Zoppello~\myenv]{Dipartimento di Scienze Matematiche ``G.~L.~Lagrange'', Politecnico di Torino, Corso Duca degli Abruzzi, 24, 10129 Torino, Italy}
\email{marta.zoppello@polito.it}

\date{\today}


\makeindex

\begin{document}

\begin{abstract}
The controllability of a fully three-dimensional $N$-link swimmer is studied. 
After deriving the equations of motion in a low Reynolds number fluid by means of Resistive Force Theory, the controllability of the minimal $2$-link swimmer is tackled using techniques from Geometric Control Theory.
The shape of the $2$-link swimmer is described by two angle parameters. It is shown that the associated vector fields that govern the dynamics generate, via taking their Lie brackets, all six linearly independent directions in the configuration space; every direction and orientation can be achieved by operating on the two shape variables.    The result is subsequently extended to the $N$-link swimmer.      Finally, the minimal time optimal control problem and the minimisation of the power expended are addressed and a qualitative description of the  optimal strategies is provided.
\end{abstract}

\maketitle

\tableofcontents

%%%%%%%%%%%%%%%%%%%%%%%%%%%%%%%%%%%%%%%%%%%%%%%%%%%%%%%%%
%%%%%%%%%%%%%%%%%%%%%%%%%%%%%%%%%%%%%%%%%%%%%%%%%%%%%%%%%
%%%%%%%%%%%%%%%%%%%%%%%%%%%%%%%%%%%%%%%%%%%%%%%%%%%%%%%%%
%%%%%%%%%%%%%%%%%%%%%%%%%%%%%%%%%%%%%%%%%%%%%%%%%%%%%%%%%
%%%%%%%%%%%%%%%%%%%%%%%%%%%%%%%%%%%%%%%%%%%%%%%%%%%%%%%%%

\section{Introduction} \label{sec:intro}
The swimming motion of microorganisms in viscous fluid at low Reynolds number has been studied mathematically since the 1950s \cite{Hancock1953,Taylor1951}. There has recently been growing interest in understanding the behaviour of simple model swimmers due to the potential to manufacture such microrobots and use them for biomedical applications \cite{Li2017,Sitti2015}. For practical reasons, it can be beneficial for a proposed robotic swimmer to be as simple as possible while achieving full controllability. Here, we define swimming to be the translational and rotational motion of the swimmer in quiescent fluid due to changes in shape of the swimmer's body; by controllability we mean the ability of prescribing the shape changes in order to steer the swimmer from a given initial configuration (\emph{i.e.}, position and orientation) to a given final one. 
We neglect gravity, assuming that the swimmer is neutrally buoyant, and in view of proposing a model for a minimal swimmer, other net forces and torques acting on the body are not considered.

It is well known for swimmers in Stokes flow that if the body undergoes a shape change that is subsequently reversed, then the swimmer would return to its original position and orientation. This result, stated by Purcell \cite{Purcell1977}, is known as the Scallop Theorem. In particular, a ``scallop'' consisting of two rigid links joined by a hinge that can open and close will not achieve any net displacement by repeatedly opening and closing its hinge. Purcell proposed that at least three links, connected by two hinges, are necessary to achieve a net displacement with periodic shape changes. This model is commonly referred to as Purcell's (planar) $3$-link swimmer, and has been shown to be controllable in two-dimensional space \cite{GMZ,Kadam2016}.

If the $3$-link swimmer is twisted so that the axes of rotation for the two hinges are perpendicular to one another, then the swimmer is no longer planar in configuration. While this variant still has only two hinges, and therefore two degrees of freedom for the shape, it was shown that this swimmer is controllable in three-dimensional space \cite{root}.

In the present work, we consider a $2$-link swimmer that has a joint with two angular degrees of freedom. This joint can be thought of as a hinge whose axis can rotate about the axis of the first link. Alternatively, this corresponds to the non-planar $3$-link swimmer in the limit that the length of the central link vanishes so that the two perpendicular hinges are next to each other. 

Note that there is a fundamental difference between the $2$-link swimmer with two degrees of motion and the non-planar $3$-link swimmer. It is clear that opening or closing either hinge changes the shape of the $3$-link swimmer. Without the central link, however, one of the hinges simply rotates a link about its axis. The shape appears indistinguishable since each link is assumed to be a cylinder with rotational symmetry. Nevertheless, we show that the $2$-link swimmer can achieve arbitrary displacements and rotations in three-dimensional space. This motion requires consideration of the viscous torque due to rotation of a link about its axis. Including this torque in the model enables the swimmer to rotate despite its shape appearing stationary due to symmetry of the cylindrical link.

\smallskip

Our $2$-link swimmer consists of a segment hinged to a thin cylindrical rod which is directed along the positive $z$-axis; the hinge is located at the origin of the co-moving reference frame, so that the shape of the swimmer is described by a point in the unit sphere $\S2$, identifying the direction of the segment (link $2$) with respect to the thin cylinder (link $1$), see Figure~\ref{fig:2link}. 
The shape parameters of the $2$-link swimmer are therefore the two angles $\vartheta$ and $\varphi$ which parametrise a point in $\S2$.
The configuration parameters are the translation $\bx\in\R3$ and rotation $R\in SO(3)$ of the change of coordinates of the co-moving frame with respect to the lab reference frame. 

By the considerations above, the only forces acting on the swimmer are the hydrodynamic ones which, due to the slenderness of the swimmer, can be accurately approximated by Resistive Force Theory \cite{GH1955}, according to which the local densities of viscous force and torque %$\bef_t(s)$ and $\btau_t(s)$ 
are linear in the components of the velocity of the swimmer which are parallel or perpendicular to the swimmer's body through suitable parallel and perpendicular drag coefficients $0<C_\parallel <C_\perp$.
Adding the viscous torque due to the rotation of link $1$ about its axis amounts to adding an extra term in the expression of the viscous torque acting on link $1$ through a torsional drag coefficient $C_\tau$ (see the expressions in \eqref{003}). 

Once the total viscous force and torque are computed, setting them equal to zero allows us to obtain the equations of motion for the swimmer. 
These are conveniently written in the form of a (nonlinear) control system, so that tools from Geometric Control Theory can be applied.
In this framework, the time changes of the shape parameters $\vartheta$ and $\varphi$ are considered as the controls $u_1,u_2$ of the system, since $\vartheta$ and $\varphi$ are the parameters that can be actuated by the swimmer to modify its shape.
Standard results and methods from Geometric Control Theory are used to prove Theorem~\ref{th:contr} ensuring controllability of the $2$-link swimmer: any given final configuration can be reached starting from any assigned initial configuration by acting on the controls $u_1,u_2$ (this is known as \emph{fiber controllability} in the sense of Definition~\ref{def:contr}(i)).
Technically, this is obtained by computing the Lie brackets of the vector fields $V_1$ and $V_2$ activated by $u_1$ and $u_2$ and showing that they generate all the possible directions of motion, thus proving that two linearly independent vectors, the $V_i$'s, generate the six-dimensional space of translations and rotations $(\bx,R)$.

Controllability for the $2$-link both ensures that the equations of motion have a unique solution (Theorem~\ref{th:exuni}) and can be easily extended to the $N$-link swimmer, providing the main result of the paper.

Controllability of the system paves the way to the study of optimal swimming strategies.
Our second result establishes the existence of an optimal solution and the qualitative characterisation of the optimal control that generates it, for two specific optimal control problems which are relevant for the applications, especially in view of possible robotic implementations. 
The minimal time optimal control problem seeks the optimal solution to move from a given configuration to another given one in the shortest possible time, whereas optimisation of the power expended deals with minimising the power expended to achieve the motion (this is useful in view in presence of limited amount of resources).
Similar optimal control problems have been tackled in \cite{DMDSM2015,GMZ,MZ} for the power expended of a filament moving on a plane, for the minimal time of a planar Purcell swimmer, and for the minimal time ad quadratic cost for a scallop subject to a switching dynamics, respectively.

\smallskip

The paper is organised as follows: in Section~\ref{sec:dyn} we describe the setting for the dynamics of the $2$-link swimmer, and we deduce the equations of motion.
In Section~\ref{sec:controllability}, we prove Theorem~\ref{th:contr} which states that the $2$-link swimmer is fiber controllable according to Definition~\ref{def:contr}(i).
We further describe how to generalise both the problem setting and the results obtained to the $N$-link swimmer in Section~\ref{sec:ext} and, in Section~\ref{sec:optimality}, we discuss some optimal control problems which are relevant in this context, namely the minimal time optimal control problem and the minimisation of the power expended. % in the spirit of \cite[Theorem~5.1]{DMDSM2015}.
Finally, Section~\ref{sec:conclusions} collects an overview of the results obtained and discusses some potential perspectives.

%\red{
%\begin{itemize}
%\item Taking $\vartheta\in\R{}$ and $\varphi\in\T$ \emph{just} makes the parametrization not injective. Is this a problem? NO, it isn't. We modified the paper accordingly.
%\item For Henry -- it might not need to be explicitly put in the final manuscript: Can the standard Purcell 3-link be described by our model? YES: we should put $N=3$ and fix the $\vartheta_t^{(i)}$'s to be all (the same) constant, so that the motion is confined to a plane. The lateral arms of the swimmer are described by two angles $\varphi_t^{(2)}$ and $\varphi_t^{(3)}$ (according to our notation), which can cross the value $\pi$ corresponding to the case when they're aligned with the central link. It's up to the tunability of the controls to devise swimming strategies such that $\varphi_t^{(i)}=0$ for $t\in(a,b)$ with $a<b$ does not happen. 
%\end{itemize}
%}

\section{Dynamics of the $2$-link swimmer} \label{sec:dyn}
Let $\bx_t$ denote the position (with respect to the lab system) of the hinge of the $2$-link swimmer, and let $\tilde \be^{(1)}_t\coloneqq \hat\be_3$ and $\tilde\be^{(2)}_t\coloneqq(\sin\varphi_t\cos\vartheta_t,\sin\varphi_t\sin\vartheta_t,\cos\varphi_t)$ be the directions of the two links, where $\vartheta_t\in\R{}$ and $\varphi_t\in\T\coloneqq\R{}/2\pi$ are the \emph{shape parameters}.\footnote{Notice that the choice of letting  $(\vartheta_t,\varphi_t)\in\R{}\times\T$ makes the parametrization of a point on the sphere $\S{2}$ not injective. This will not affect the description of the motion of the swimmer.} %; notice that $\tilde\be^{(2)}_t\in\S{2}\setminus\{\text{north pole, south pole}\}$ if we take $\vartheta_t\in[0,2\pi)$ and $\varphi_t\in(0,\pi)$: this prevents degenerate configurations of the $2$-link. % for $i\in\{1,2\}$. 
Finally, let $\ell_i$ be the length of the $i$-th link, so that the generic point on the $i$-th link, at a distance $s\in[0,\ell_i]$ from $\bx_t$ is given by $s\tilde\be^{(i)}_t$, since in the co-moving system the hinge is located at the origin.

\tdplotsetmaincoords{60}{110}

%
% \thetavec e \phivec sono invertite nelle prime 2 figure
% \thetavect e \phivect sono corrette
\pgfmathsetmacro{\rvec}{.8}
\pgfmathsetmacro{\thetavec}{30}
\pgfmathsetmacro{\phivec}{60}
\pgfmathsetmacro{\rvect}{.3}
\pgfmathsetmacro{\thetavect}{60}
\pgfmathsetmacro{\phivect}{120}

\begin{figure}[h]
	\centering
\begin{tikzpicture}[scale=5,tdplot_main_coords]
\coordinate (O) at (0,0,0);
\draw[thin,->] (0,0,0) -- (1,0,0) node[anchor=north east]{$x$};
\draw[thin,->] (0,0,0) -- (0,1,0) node[anchor=north west]{$y$};
\draw[thin,->] (0,0,0) -- (0,0,1) node[anchor=south]{$z$};
\tdplotsetcoord{P}{\rvec}{\thetavec}{\phivec}
\tdplotsetcoord{P/2}{\rvec /2}{\thetavec}{\phivec}

\foreach \t in {10,20,...,360}% generatrices
		\draw[green] ({0.03*cos(\t)},{0.03*sin(\t)},0)
      --({0.03*cos(\t)},{0.03*sin(\t)},{0.69});
	\draw[green,very thick] (0,0,0) % lower circle
		\foreach \t in {5,10,...,360}
			{--({0.03*cos(\t)},{0.03*sin(\t)},0)};
	\draw[green,very thick] (0,0,{0.69}) % upper circle
		\foreach \t in {10,20,...,360}
			{--({0.03*cos(\t)},{0.03*sin(\t)},{0.69})};

\draw[green,thick] (O) -- (P);
\draw[dashed] (O) -- (Pxy);
\draw[dashed] (P) -- (Pxy);
\draw[thick] (O) -- (Pz);
\draw [ fill ] (O) circle [ radius = 0.2 pt];
\draw [green, fill ] (P) circle [ radius = 0.2 pt];
\draw [green, fill ] (Pz) circle [ radius = 0.2 pt];
\draw [ultra thick, ->] (O) -- (P/2) node[right] {$\tilde{\mathbf{e}}_t^{(2)}$};
\draw [ultra thick, ->] (O) -- (P/2z) node[left] {$\tilde{\mathbf{e}}_t^{(1)} \equiv \hat{\mathbf{e}}_3$};
\draw [ultra thick, ->] (O) -- (0.35,0,0) node[left] {$\hat{\mathbf{e}}_1$};
\draw [ultra thick, ->] (O) -- (0,0.3,0) node[above right] {$\hat{\mathbf{e}}_2$};
\draw [thin] (O) -- (O) node[left] {$\mathbf{x}_t$};
\tdplotdrawarc{(O)}{0.2}{0}{\phivec}{anchor=north}{$\theta_t$}
\tdplotsetthetaplanecoords{\phivec}
\tdplotdrawarc[tdplot_rotated_coords]{(0,0,0)}{0.5}{0}
{\thetavec}{anchor=south west}{$\varphi_t$}
\end{tikzpicture}
\caption{Co-moving frame of the $2$-link swimmer.}
\label{fig:2link}
\end{figure}
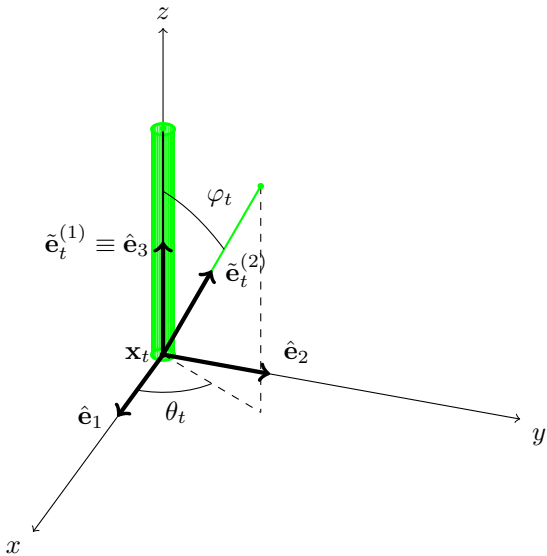

In the lab system, the hinge located at $\bx_t$ is also rotated by a rotation matrix $R_t\in SO(3)$, so that, denoting by
\begin{equation}\label{001}
\be^{(i)}_t\coloneqq R_t\tilde\be^{(i)}_t,\qquad i\in\{1,2\}
\end{equation}
the directions of the links, the generic point $\bx^{(i)}_t(s)$ on link $i$ at a distance $s$ from $\bx_t$ is identified by
\begin{equation}\label{002}
\bx^{(i)}_t(s)=\bx_t+s\be^{(i)}_t=\bx_t+sR_t\tilde\be^{(i)}_t.
\end{equation}

The densities of viscous force $\bef^{(i)}_t(s)$ and torque $\btau^{(i)}_t(s)$ are computed using \emph{Resistive Force Theory} by
\begin{subequations}\label{003}
\begin{align}
\bef^{(i)}_t(s)= &\, [(C_\parallel-C_\perp)\be^{(i)}_t\otimes\be^{(i)}_t+C_\perp I]\, \dot \bx^{(i)}_t(s), \label{003f} \\
\btau^{(1)}_t(s)= &\, s\be^{(1)}_t \times \bef^{(1)}_t(s)+C_\tau( \be^{(1)}_t\otimes\be^{(1)}_t) \bomega_t, \label{003t1}\\
\btau^{(2)}_t(s)= &\, s\be^{(2)}_t \times \bef^{(2)}_t(s), \label{003t2}%\red{+ \text{other term due to fatness of the link?}} \label{003t}
\end{align}
\end{subequations}
where $C_\parallel$ and $C_\perp$ are the parallel and perpendicular drag coefficients to each link and $C_\tau$ is the torsional drag coefficient which takes into account the fact that the first link is a cylinder, the symbol $\otimes$ denotes the dyadic product of vectors ($(\ba\otimes\bb)_{ij}\coloneqq a_ib_j$), the symbol $\times$ denotes the vector product in $\R{3}$, and a superimposed dot denotes derivation with respect to time.

In order to compute the expressions in \eqref{003}, we need to take the time derivative $\dot\bx^{(i)}_t(s)$, which, by \eqref{001} and \eqref{002}, reads
\begin{equation}\label{004}
\begin{split}
\dot \bx^{(i)}_t(s)=\dot \bx_t+s\dot\be^{(i)}_t= & \,\dot\bx_t+s\dot R_t\tilde\be^{(i)}_t+sR_t\dot{\tilde\be}^{(i)}_t \\
=&\, \dot\bx_t+s\dot R_t R^{-1}_t R_t \tilde\be^{(i)}_t+sR_t\dot{\tilde\be}^{(i)}_t \\
=&\, \dot\bx_t+s\Omega_t  \be^{(i)}_t+sR_t\dot{\tilde\be}^{(i)}_t = \dot\bx_t+s\bomega_t \times \be^{(i)}_t+sR_t\dot{\tilde\be}^{(i)}_t \,,
\end{split}
\end{equation}
where $\Omega_t$ and $\bomega_t$ are the angular matrix and the angular velocity, respectively, associated with the rotation matrix $R_t$.

Taking some elementary vector identities\footnote{$(\ba\otimes\bb)\bv=(\bb\cdot\bv)\ba$; $(\bomega\times\be)\cdot\be=0$.} into account, we obtain
\begin{equation}\label{005}
\begin{split}
\bef^{(1)}_t(s)= &\, [(C_\parallel-C_\perp)\be^{(1)}_t\otimes\be^{(1)}_t+C_\perp I]\dot \bx_t+sC_\perp \bomega_t\times \be^{(1)}_t\,, \\
\bef^{(2)}_t(s)= &\, [(C_\parallel-C_\perp)\be^{(2)}_t\otimes\be^{(2)}_t+C_\perp I]\dot \bx_t+sC_\perp \bomega_t\times \be^{(2)}_t+sC_\perp R_t\dot{\tilde\be}^{(2)}_t\,,
\end{split}
\end{equation}
where we also used that $\lvert \tilde\be^{(i)}_t\rvert\equiv1$ implies that $\tilde\be^{(i)}_t\cdot\dot{\tilde\be}^{(i)}_t=0$.
Moreover,
\begin{equation*}%\label{006}
\begin{split}
\btau^{(1)}_t(s)= &\, sC_\perp \be^{(1)}_t\times \dot \bx_t+s^2C_\perp[I-\be^{(1)}_t\otimes\be^{(1)}_t]\bomega_t+C_\tau( \be^{(1)}_t\otimes\be^{(1)}_t) \bomega_t\,, \\
\btau^{(2)}_t(s)= &\, sC_\perp \be^{(2)}_t\times \dot \bx_t+s^2C_\perp[I-\be^{(2)}_t\otimes\be^{(2)}_t]\bomega_t+s^2C_\perp \be^{(2)}_t\times R_t\dot{\tilde\be}^{(2)}_t\,.
\end{split}
\end{equation*}
integrating from $0$ to $\ell_i$ with respect to $s$, we obtain
\begin{equation*}%\label{007}
\begin{split}
\bF^{(1)}_t =&\, [(C_\parallel-C_\perp)\be^{(1)}_t\otimes\be^{(1)}_t+C_\perp I]\ell_1\dot \bx_t+\frac{\ell_1^2}{2}C_\perp \bomega_t\times \be^{(1)}_t \\
=&\, \ell_1 R_t [(C_\parallel-C_\perp)\tilde\be^{(1)}_t\otimes\tilde\be^{(1)}_t+C_\perp I] R_t^{-1} \dot \bx_t -\frac{\ell_1^2}{2}C_\perp R_t [\tilde\be^{(1)}_t\times (R_t^{-1}\bomega_t)], \\
\bF^{(2)}_t= &\, [(C_\parallel-C_\perp)\be^{(2)}_t\otimes\be^{(2)}_t+C_\perp I]\ell_2\dot \bx_t+\frac{\ell_2^2}{2}C_\perp \bomega_t\times \be^{(2)}_t+\frac{\ell_2^2}{2}C_\perp R_t\dot{\tilde\be}^{(2)}_t \\
=&\, \ell_2 R_t [(C_\parallel-C_\perp)\tilde\be^{(2)}_t\otimes\tilde\be^{(2)}_t+C_\perp I] R_t^{-1} \dot \bx_t -\frac{\ell_2^2}{2}C_\perp R_t [\tilde\be^{(2)}_t\times (R_t^{-1}\bomega_t)] +\frac{\ell_2^2}{2}C_\perp R_t\dot{\tilde\be}^{(2)}_t, \\
\bT^{(1)}_t= &\, \frac{\ell_1^2}{2}C_\perp \be^{(1)}_t\times \dot \bx_t+\frac{\ell_1^3}{3}C_\perp[I-\be^{(1)}_t\otimes\be^{(1)}_t]\bomega_t+\ell_1C_\tau( \be^{(1)}_t\otimes\be^{(1)}_t) \bomega_t\\
=&\, \frac{\ell_1^2}{2}C_\perp R_t [ \tilde\be^{(1)}_t\times (R_t^{-1} \dot \bx_t)]+\frac{\ell_1^3}{3}C_\perp R_t[I-\tilde\be^{(1)}_t\otimes\tilde\be^{(1)}_t] R_t^{-1}\bomega_t \\
& +\ell_1C_\tau R_t(\tilde{ \be}^{(1)}_t\otimes\tilde{\be}^{(1)}_t) R_t^{-1}\bomega_t, \\
\bT^{(2)}_t= &\, \frac{\ell_2^2}{2}C_\perp \be^{(2)}_t\times \dot \bx_t+\frac{\ell_2^3}{3}C_\perp[I-\be^{(2)}_t\otimes\be^{(2)}_t]\bomega_t+\frac{\ell_2^3}{3}C_\perp \be^{(2)}_t\times R_t\dot{\tilde\be}^{(2)}_t \\
=&\,  \frac{\ell_2^2}{2}C_\perp R_t [ \tilde\be^{(2)}_t\times (R_t^{-1} \dot \bx_t)]+\frac{\ell_1^3}{3}C_\perp R_t[I-\tilde\be^{(2)}_t\otimes\tilde\be^{(2)}_t] R_t^{-1}\bomega_t + \frac{\ell_2^3}{3}C_\perp R_t (\tilde\be^{(2)}_t\times \dot{\tilde\be}^{(2)}_t)
\end{split}
\end{equation*}
The total viscous force is then given by
\begin{equation}\label{008}
\bF_t=\bF^{(1)}_t+\bF^{(2)}_t= R_t\tilde K_tR_t^{-1}\dot\bx_t + R_t \tilde C_t^\top R_t^{-1}\bomega_t+R_t \tilde\bF^\sh_t
\end{equation}
and the total viscous torque by
\begin{equation}\label{011}
\bT_t=\bT^{(1)}_t+\bT^{(2)}_t= R_t \tilde C_t R_t^{-1} \dot\bx_t+ R_t \tilde J_t R_t^{-1}\bomega_t+R_t \tilde\bT^\sh_t\,,
\end{equation}
where the matrices $K_t$, $C_t$, and $J_t$ are defined by
\begin{equation}\label{009}
\begin{split}
\tilde K_t\coloneqq &\, \tilde K^{(1)}_t+\tilde K^{(2)}_t,\quad\text{with}\quad \tilde K^{(i)}_t\coloneqq [(C_\parallel-C_\perp)\tilde\be^{(i)}_t\otimes\tilde\be^{(i)}_t+C_\perp I]\ell_i \,,\\
\tilde C_t\coloneqq &\, \tilde C^{(1)}_t+\tilde C^{(2)}_t,\quad\text{with}\quad \tilde C^{(i)}_t\coloneqq \frac{\ell_i^2}{2} C_\perp \tilde E^{(i)}_t\quad\text{and $\tilde E^{(i)}_t$ such that}\;\tilde E^{(i)}_t\bv=\tilde \be^{(i)}_t\times\bv, \\
\tilde J_t\coloneqq &\, \tilde J^{(1)}_t+\tilde J^{(2)}_t,\quad\text{with}\quad \tilde J^{(1)}_t\coloneqq \frac{\ell_1^3}{3}C_\perp[I-\tilde\be^{(1)}_t\otimes\tilde\be^{(1)}_t]+ \ell_1C_\tau \tilde\be^{(1)}_t\otimes\tilde\be^{(1)}_t\quad\text{and}\\
&\phantom{ \tilde J^{(1)}_t+\tilde J^{(2)}_t,\quad\text{with}}\quad  \tilde J^{(2)}_t\coloneqq \frac{\ell_2^3}{3}C_\perp[I-\tilde\be^{(2)}_t\otimes\tilde\be^{(2)}_t],
\end{split}
\end{equation}
and the viscous force and torque due to the shape deformation are
\begin{equation*}%\label{010}
\tilde\bF_t^\sh\coloneqq \frac{\ell_2^2}{2} C_\perp  \dot{\tilde\be}_t^{(2)}\qquad\text{and}\qquad \tilde\bT_t^\sh\coloneqq \frac{\ell_2^3}{3}C_\perp  \tilde E^{(2)}_t \dot{\tilde\be}^{(2)}_t\,.
\end{equation*}
Expressions \eqref{008} and \eqref{011} can be written together in matricial form as
\begin{equation}\label{012}
\!\!\!\!
\vec{\bF_t}{\bT_t}=\mat{R_t}{0}{0}{R_t}\mat{\tilde K_t}{\tilde C_t^\top}{\tilde C_t}{\tilde J_t}\mat{R_t^{-1}}{0}{0}{R_t^{-1}}\vec{\dot\bx_t}{\bomega_t}+\mat{R_t}{0}{0}{R_t}\vec{\tilde\bF^\sh_t}{\tilde\bT^\sh_t}.
\end{equation}
The matrix
\begin{equation}\label{013}
\tilde \cM_t\coloneqq \mat{\tilde K_t}{\tilde C_t^\top}{\tilde C_t}{\tilde J_t}
\end{equation}
is known in the literature as the \emph{grand resistance matrix}. 
It is a $6\times 6$ symmetric (see \eqref{009}) and positive-definite (see \cite{HappelBrenner}) %\red{MM: should we prove it? I didn't check it, but it might be easy to obtain} 
matrix.

Suppose that the two links are of equal lengths, $l_1=l_2\eqqcolon L$.
%\mm{MM: In fact, it is only $\ell_2$, so that $L=\ell_2$}
Listing also $\dot\varphi_t$ and $\dot\vartheta_t$ in the state of the system, and setting \eqref{012} equal to zero (this is sometimes called the \emph{self-propulsion constraint}), we have
\begin{equation}\label{014}
\left(
\begin{array}{c}
R_t^{-1}\dot\bx_t\\
R_t^{-1}\bomega_t\\
\dot\varphi_t\\
\dot\vartheta_t
\end{array}
\right)= V_1(\vartheta_t,\varphi_t)u_1+V_2(\vartheta_t,\varphi_t)u_2\,,
\end{equation}
where
\begin{equation}\label{021}
\!\!\!\! V_1\coloneqq 
\left(\begin{array}{c}
\!\! \tilde \cM_t^{-1}\left(
\begin{array}{c}
-\frac{L^2}2 C_\perp\cos\vartheta_t\cos\varphi_t\\
-\frac{L^2}2 C_\perp\sin\vartheta_t\cos\varphi_t\\
\frac{L^2}2 C_\perp\sin\varphi_t\\
\frac{L^3}3 C_\perp\sin\vartheta_t\\
-\frac{L^3}3 C_\perp\cos\vartheta_t\\
0
\end{array}\right)\\
1\\
0
\end{array}
\right),\,
V_2\coloneqq
\left(
\begin{array}{c}
\!\! \tilde \cM_t^{-1}\left(\begin{array}{c}
\frac{L^2}2 C_\perp\sin\vartheta_t\sin\varphi_t\\
-\frac{L^2}2 C_\perp\cos\vartheta_t\sin\varphi_t\\
0\\
\frac{L^3}6 C_\perp\cos\vartheta_t\sin2\varphi_t\\
\frac{L^3}6 C_\perp\sin\vartheta_t\sin2\varphi_t\\
-\frac{L^3}3 C_\perp\sin^2\varphi_t
\end{array}\right)\\
0\\
1
\end{array}
\right)
\end{equation} 
and $u_1,u_2\colon[0,T]\to\R{}$ are measurable functions.
By straightforward computations we have
\begin{equation}\label{015}
V_1=\left(
\begin{array}{c}
\displaystyle \frac{LC_\perp\cos\vartheta_t\sin^2\frac{\varphi_t}{2}}{2(C_\perp+C_\parallel+(C_\parallel-C_\perp)\cos\varphi_t)}\\
\displaystyle \frac{LC_\perp\sin\vartheta_t\sin^2\frac{\varphi_t}{2}}{2(C_\perp+C_\parallel+(C_\parallel-C_\perp)\cos\varphi_t)}\\
\displaystyle \frac{LC_\perp\sin\varphi_t}{4(C_\perp+C_\parallel+(C_\parallel-C_\perp)\cos\varphi_t)}\\
\displaystyle \frac{\sin\vartheta_t}{2}\\
\displaystyle -\frac{\cos\vartheta_t}{2}\\
0\\
1\\
0
\end{array}
\right)\quad\text{and}
\end{equation}
\begin{equation}\label{016N}
V_2=\left(
\begin{array}{c}
\displaystyle \frac{-24 C_\tau  L \sin\vartheta_t  \sin ^2\frac{\varphi_t }{2} \sin \varphi _t}{36 C_\tau  \cos \varphi_t -45 C_\tau +\cos 2 \varphi_t  \left(2 C_\perp  L^2-15 C_\tau \right)-2 C_\perp  L^2}\\
\displaystyle \frac{48 C_\tau  L \cos \vartheta_t  \sin ^3\frac{\varphi_t }{2} \cos \frac{\varphi_t }{2}}{36 C_\tau  \cos \varphi_t -45 C_\tau +\cos 2 \varphi_t  \left(2 C_\perp  L^2-15 C_\tau
   \right)-2 C_\perp  L^2}\\
   0\\
   \displaystyle \frac{-3 C_\tau  \cos \vartheta_t  (5 \sin 2 \varphi_t -6 \sin \varphi_t )}{36 C_\tau  \cos \varphi_t -45 C_\tau +\cos 2 \varphi_t  \left(2 C_\perp  L^2-15 C_\tau \right)-2 C_\perp  L^2}\\
    \displaystyle \frac{-3 C_\tau  \sin \vartheta_t  (5 \sin 2 \varphi_t -6 \sin \varphi_t )}{36 C_\tau  \cos \varphi_t -45 C_\tau +\cos 2 \varphi_t  \left(2 C_\perp  L^2-15 C_\tau \right)-2 C_\perp  L^2}\\
    \displaystyle \frac{4L^2C_\perp\sin^2\varphi_t}{36 C_\tau  \cos \varphi_t -45 C_\tau +\cos 2 \varphi_t  \left(2 C_\perp  L^2-15 C_\tau \right)-2 C_\perp  L^2}\\
    0\\
    1
\end{array}
\right).
%V_2=\left(
%\begin{array}{c}
%0\\
%0\\
%0\\
%0\\
%0\\
%-1\\
%0\\
%1
%\end{array}
%\right).
\end{equation}
The following theorem, whose proof is a byproduct of the controllability Theorem~\ref{th:contr}, holds.
\begin{theorem}\label{th:exuni}
Let $(\bar\bx,\bar R)\in\R{3}\times SO(3)$ be given.
There exists a unique absolutely continuous solution $(\bx_t,R_t)\colon[0,+\infty)\to \R{3}\times SO(3)$ to the Cauchy problem for \eqref{014} with initial condition $(\bx_0,R_0)=(\bar\bx,\bar R)$, for any controls $u_1,u_2\in L^\infty(0,+\infty)$.
\end{theorem}

\section{Controllability}\label{sec:controllability}

\subsection{Preliminaries}\label{sec:prel}

In this subsection we present the basic notions about control systems on Lie groups.
We use their properties in order to state the controllability results for the $2$-link swimmer in Subsection~\ref{sec:contrthm}.

Let $G$ be an $n$-dimensional matrix Lie group %, whose generic element $g$ denotes the spatial position and orientation of the swimmer, 
and let  $\cS$ be an $m$-dimensional parallelizable manifold (see \cite[page 160]{BishopGoldberg});
%\red{(i.e., an open set in $\R{m}$\footnote{maybe not the best way to say what it is, check better})}; %, whose generic element $s$ describes a shape of the swimmer.
we call $\bM\coloneqq G\times \cS$ the \emph{configuration space}, whose generic element is $\bz\coloneqq (g,s)$.
\begin{defin}\label{def:controlsystem}
A \emph{nonlinear control system on $G$} is an ODE of the form
\begin{equation}\label{000}
\dot\bz=\begin{pmatrix}
\dot g \\ \dot s
\end{pmatrix} = \begin{pmatrix}
g \xi(s,u) \\
u
\end{pmatrix},
\end{equation}
where $\xi$ is a map from the tangent space $T\cS$ to the Lie algebra $\fg$ of $G$ which is linear in the fibers, i.e.,
$$\xi(s,u)=\sum_{i=1}^m \xi_i(s)u_i,\qquad\text{for some analytic (nonlinear) maps $\xi_i\colon\cS\to\fg$,\quad $i=1,\ldots,m$,}$$
and $u\colon[0,T]\to(u_1(t),\ldots,u_m(t))\in T_s\cS\simeq\R{m}$ is the vector of controls.
\end{defin}
Denoting by $\hat\be_i^{\R{m}}$ the elements of the canonical basis of $\R{m}$, system \eqref{000} can be written as
\begin{equation}\label{000_4}
\dot\bz=\sum_{i=1}^m \begin{pmatrix}
g\xi_i(s) \\
\hat \be_i^{\R{m}}
\end{pmatrix} u_i\eqqcolon \sum_{i=1}^m Z_i(s)u_i, 
\end{equation}
where $Z_i=(Z_i^G,Z_i^\cS)\colon \cS\to T_g G\times T_s \cS\simeq T_g G\times \R{m}$, for $i\in\{1,\ldots,m\}$.
\begin{defin}\label{def:equivariant}
Let $g\in G$.
A vector field $X$ on $\bM$ is \emph{equivariant} with respect to the group action
\begin{equation}\label{000_5}
\Psi_{g}\colon \bM\to\bM, \qquad \bz=(h,s)\mapsto\Psi_{g}(\bz)\coloneqq(g h,s)
\end{equation}
if, denoting by $(\cdot)_*$ the push-forward,
\begin{equation}\label{000_6}
(\Psi_{g})_* X(\bz)=X\big(\Psi_{g}(\bz)\big), \qquad\text{for $\bz=(h,s)\in\bM$}.
\end{equation}
\end{defin}
By the definition of push-forward, the left-hand side in \eqref{000_6} is $\big((D\Psi_g)(\Psi_g^{-1}(\bz))\big)\cdot X(\Psi_g^{-1}(\bz))$, where $D$ denotes the differential; since $\Psi_g$ defined in \eqref{000_5} is nothing but the left-translation by $g$ in the $G$-component of $\bz$, it turns out that 
\begin{equation*}%\label{000_7}
\big((D\Psi_g)(\Psi_g^{-1}(\bz))\big)=\begin{pmatrix}
T_e L_g & 0 \\
0 & I_m 
\end{pmatrix}=
\begin{pmatrix}
g & 0 \\
0 & I_m 
\end{pmatrix}, 
\end{equation*}
where $L_g$ is the left translation by $g\in G$ (namely, $L_g h=gh$), $T_e$ is the tangent map to the identity $e\in G$, and $I_m$ is the $m$-dimensional identity matrix.

\begin{remark}\label{remark}
The following observations are straightforward:
\begin{itemize}
\item[(i)] for any $\bar g\in G$, the vector fields $Z_i$ ($i=1,\ldots,m$) in \eqref{000_4} are equivariant with respect to the group action $\Psi_{\bar g}$ defined in \eqref{000_5};
\item[(ii)] for any $Z_i,Z_j\in T_gG\times \R{m}$ and for any $\bar g\in G$, the Lie bracket $[Z_i,Z_j]$ is equivariant with respect to the group action $\Psi_{\bar g}$.
%\begin{equation}\label{052}
%\cL^{-1}\big([Z_1^G,Z_2^G]_{\fse(3)}+(\nabla_s Z_2^G)Z_1^\cS-(\nabla_s Z_1^G)Z_2^\cS \big)=[\cL^{-1}Z_1^G,\cL^{-1}Z_2^G]_{\R{6}}\,, 
%\end{equation}
%where the Lie brackets $[\xi_1,\xi_2]_{\fse(3)}$ is nothing but the matrix commutator $\xi_1\xi_2-\xi_2\xi_1$, whereas the Lie brackets in the right-hand side are defined as 
%$$[v_1,v_2]_{\R{6}}\coloneqq (\nabla_\bz v_2)v_1-(\nabla_\bz v_1)v_2.$$
\end{itemize}
\end{remark}
We now give the definition of fiber controllability and controllability.
%Consider a non-linear control system with two control inputs:
%\begin{equation}
%\label{000}
%\begin{aligned}
%&\dot\bz=f(\bz,\bu)=h_1(\bz)u_1+h_2(\bz)u_2\\
%&\bz\in\bM\\
%&h_1,h_2:\bM\rightarrow T_z\bM\\
%&u_1,u_2:[0.T]\rightarrow \mathbb{R}
%\end{aligned}
%\end{equation}
%with $\bM$ an $n$ dimensional manifold.\\
%Suppose now that $\bM=G\times\cS$ where $G$ is a matrix Lie group and $\cS$ a $2-$dimensional manifold usually called \textit{shape space}. This is trivially a principal fiber bundle with the group $G$ acting on itself. Let us assume to be able to control all the velocities of the two variables in $\cS$. Under these assumptions the vector fields $h_i$ $i=1,2$ are $G$ invariant and the system \eqref{000} can be rewritten 
%\begin{equation}
%\label{000_1}
%\begin{pmatrix}\dot g\\
%\dot s_1\\\dot s_2\end{pmatrix}=\begin{pmatrix}g \tilde h_1(s)\\1\\0\end{pmatrix} u_1+\begin{pmatrix}g \tilde h_2(s)\\0\\1\end{pmatrix} u_2
%\end{equation}
\begin{defin}\label{def:contr}
The nonlinear control system \eqref{000}
\begin{itemize} 
\item[(i)] %The non-linear control system \eqref{000_1} 
is said to be \emph{fiber controllable} if for any initial $(g_0,s_0)\in \bM$ and final $g_1\in G$ there exist a time $T>0$ and control inputs $u\colon[0,T]\to \R{m}$ such that $g(0)=g_0$ and $g(T)=g_1$, where $(g(t),s(t))$ is the unique solution to \eqref{000}.
\item[(ii)] %The nonlinear control system \eqref{000} 
is said to be \emph{fiber controllable at $(g_0,s_0)\in \bM$} if there exists a neighbourhood $U_{g_0}$ of $g_0\in G$ such that for each $g_1\in U_{g_0}$ there exist a time $T>0$ and control inputs $u\colon[0,T]\to \R{m}$ such that $g(0)=g_0$ and $g(T)=g_1$, where $(g(t),s(t))$ is the unique solution to \eqref{000}.
%\item[(ii)] %The non-linear control system \eqref{000_1} 
%is said to be \emph{controllable} if for any initial $(g_0,s_0)\in\bM$ and final $(g_1,s_1)\in \bM$ there exist a time $T>0$ and control inputs $u\colon[0,T]\to \R{m}$ such that $(g(0),s(0))=(g_0,s_0)$ and $(g(T),s(T))=(g_1,s_1)$ where $(g(t),s(t))$ is the unique solution to \eqref{000}.
\end{itemize}
\end{defin}
It can immediately be noted that if condition (ii) in Definition~\ref{def:contr} holds for every $(g_0,s_0)\in\bM$, then condition (i) holds.
We observe that the uniqueness of solutions to \eqref{000} is granted by \cite[Lemma~2.1]{JS}.
\begin{theorem}\label{ostro}
Let us consider a control system on a Lie group $G$ of the form  \eqref{000_4}. Then
\begin{itemize}
\item[(i)] it is fiber controllable at $(g_0,s_0)\in\bM$ in the sense of Definition~\ref{def:contr}(ii) 
if 
\begin{equation}\label{condition}
\Pi_G \big(\fLie(\{Z_1,\ldots,Z_m\})_{(e,s_0)}\big)=\fg,
\end{equation} 
where $\fLie(\{Z_1,\ldots,Z_m\})$ is the Lie algebra generated by the vector fields $Z_1,\ldots,Z_m$ and $\Pi_G$ denotes the projection on the group component;
\item[(ii)] if condition \eqref{condition} hold for every $s_0\in \cS$, then it %the system \eqref{000_4} 
is fiber controllable in the sense of Definition~\ref{def:contr}(i).
\end{itemize}
\end{theorem}
\begin{proof}
(i)
The proof is a straightforward application of \cite[Theorem~5.9]{CMOZ2002}, where it is proved that condition \eqref{condition} implies \emph{local fiber configuration accessibility at $(g_0,s_0)$} (see \cite[Definition~5.7]{CMOZ2002}) for affine control systems on Lie groups.
Since it is well known that for driftless systems accessibility is equivalent to controllability and that the result holds globally in time, fiber controllability at $(g_0,s_0)$ in the sense of Definition~\ref{def:contr}(ii) follows.

(ii) This is easily proved since condition \eqref{condition} is independent of $g_0$.
\end{proof}

The following statement of the Orbit Theorem can be easily derived from  \cite[Chapter~2, Theorems~1 and ~2]{jurdjevic}.
\begin{theorem}[The orbit theorem]\label{orbit}
Let $\bM$ be an analytic manifold, and let $\cZ$ be a family of analytic vector fields on $\bM$. 
Then
\begin{itemize}
\item[(a)] each orbit of $\cZ$ is an analytic submanifold of $\bM$, and 
\item[(b)] if $\bN$ is an orbit of $\cZ$, then the tangent space of $\bN$ at $\bz$ is given by $\mathfrak{Lie}_\bz(\cZ)$.
In particular, the dimension of $\mathfrak{Lie}_\bz(\cZ)$ is constant as $\bz$ varies on $\bN$.
\end{itemize}
\end{theorem}

\subsection{The controllability theorem}\label{sec:contrthm}
We are interested in studying how the shape change of our swimmer determines its spatial position and orientation in the framework of control systems on Lie groups.
We will work with $\bM=G\times\cS=SE(3)\times \big(\R{}\times\T\big)$, by posing
\begin{equation}\label{050}
g\coloneqq 
\begin{pmatrix}
R(\alpha,\beta,\gamma) & \tau \\
0 & 1
\end{pmatrix}\in SE(3)\quad\text{and $s\coloneqq(\vartheta,\varphi)\in \R{}\times\T$,}
\end{equation}
where $R(\alpha,\beta,\gamma)\in SO(3)$ and $\tau\coloneqq(x_1,x_2,x_3)^\top\in\R{3}$. 
In order to write system \eqref{000_4} in vector form, we introduce the Lie algebra isomorphism $\cL\colon\R{6}\to\fse(3)$ defined by
$$y=(y_1,\ldots,y_6)^\top\mapsto
\begin{pmatrix}
0 & -y_6 & y_5 & y_1 \\
y_6 & 0 & -y_4 & y_2 \\
-y_5 & y_4 & 0 & y_3 \\
0 & 0 & 0 & 0
\end{pmatrix}.
$$
The application of $\cL^{-1}$ to the $g$-component in \eqref{000_4} will transform it from a $4\times 4$-matrix into a vector in $\R{6}$.
Moreover, denoting by $Z^G$ and $Z^\cS$ the $G$- and $\cS$-components, respectively, of any $Z\in T_gSE(3)\times\R{2}$, Remark~\ref{remark}(ii) implies that, for any $Z_1,Z_2\in T_g SE(3)\times\R{2}$, 
\begin{equation}\label{055}
(\Psi_g^{-1})_* [Z_1,Z_2]_{Tg SE(3)\times\R{2}}= \big[(\Psi_g^{-1})_* Z_1, (\Psi_g^{-1})_* Z_2\big]_{\fse(3)\times\R{2}}.
\end{equation}
Moreover, since $\cL$ is a Lie algebra isomorphism, if $Z_i=(g\xi_i(s),\hat\be_{i}^{\R{2}})$, $i=1,2$, we can rewrite \eqref{055} as
\begin{equation*}%\label{056}
\begin{split}
\begin{pmatrix}
\cL^{-1} \Gamma^G \\
\Gamma^\cS
\end{pmatrix}= & \left[
\begin{pmatrix}
\cL^{-1} \big( (\Psi_g^{-1})_*Z_1\big)^G \\
\big( (\Psi_g^{-1})_*Z_1\big)^\cS
\end{pmatrix}, 
\begin{pmatrix}
\cL^{-1} \big( (\Psi_g^{-1})_*Z_2\big)^G \\
\big( (\Psi_g^{-1})_*Z_2\big)^\cS
\end{pmatrix}
\right]_{\R{8}} \\ %{\fse(3)\times\R{2}} \\
= & 
\begin{pmatrix}
\cL^{-1}\big([\xi_1,\xi_2]_{\fse(3)} +(\nabla_s \xi_2)\hat\be_1^{\R{2}}-(\nabla_s \xi_1)\hat\be_2^{\R{2}}\big) \\
0\\
0
\end{pmatrix}, % \in\R{8} %\\
%= & \left[
%\begin{pmatrix}
%\cL^{-1}\xi_1 \\
%\hat\be_1^{\R{2}}
%\end{pmatrix},
%\begin{pmatrix}
%\cL^{-1}\xi_2 \\
%\hat\be_2^{\R{2}}
%\end{pmatrix}
%\right]_{\R{8}}\, ,
\end{split}
\end{equation*}
where we have denoted by $\Gamma$ the left-hand side in \eqref{055}.
%for any $Z_1,Z_2\in T_g SE(3)\times\R{2}$, $Z_i=(g\xi_i(s),\hat\be_{i}^{\R{2}})$, $i=1,2$.
We recall here that $[\xi_1,\xi_2]_{\fse(3)}=\xi_1\xi_2-\xi_2\xi_1$ is the commutator, for any $\xi_1,\xi_2\in\fse(3)$.
%\begin{remark}\label{remark}
%The following observations are straightforward:
%\begin{itemize}
%\item[(i)] the vector fields $Z_i$ ($i=1,\ldots,m$) in \eqref{000_4} are equivariant with respect to $\Psi_g$ defined in \eqref{000_5};
%\item[(ii)] the fact that $\cL$ is a Lie algebra isomorphism implies that, for any $Z_1,Z_2\in T_gG\times \R{m}$,
%\begin{equation}\label{052}
%\cL^{-1}\big([Z_1^G,Z_2^G]_{\fse(3)}+(\nabla_s Z_2^G)Z_1^\cS-(\nabla_s Z_1^G)Z_2^\cS \big)=[\cL^{-1}Z_1^G,\cL^{-1}Z_2^G]_{\R{6}}\,, 
%\end{equation}
%where the Lie brackets $[\xi_1,\xi_2]_{\fse(3)}$ is nothing but the matrix commutator $\xi_1\xi_2-\xi_2\xi_1$, whereas the Lie brackets in the right-hand side are defined as 
%$$[v_1,v_2]_{\R{6}}\coloneqq (\nabla_\bz v_2)v_1-(\nabla_\bz v_1)v_2.$$
%\end{itemize}
%\end{remark}

We can now state the controllability theorem for the $2$-link swimmer.
\begin{theorem}[Controllability of the $2$-link]\label{th:contr}
The $2$-link swimmer is fiber controllable in the sense of Definition~\ref{def:contr}(i) for almost any choice of the parameters $(C_\parallel, C_\perp, C_\tau, L)\in (0,+\infty)^4$.
\end{theorem}
%A basis of the Lie algebra $\fse(3)$ is explicitly constructed.
\begin{proof}
The proof is divided into three steps.

\noindent \textit{Step 1.} By \eqref{050},
%\red{$\bM=G\times\cS=SE(3)\times \big([0,2\pi)\times(0,\pi)\big)$, letting
%\begin{equation}\label{050}
%g\coloneqq 
%\begin{pmatrix}
%R(\alpha,\beta,\gamma) & \tau \\
%0 & 1
%\end{pmatrix}\in SE(3)\quad\text{and $s\coloneqq(\vartheta,\varphi)\in [0,2\pi)\times(0,\pi)$,}
%\end{equation}
%where $R(\alpha,\beta,\gamma)\in SO(3)$ and $\tau\coloneqq(x_1,x_2,x_3)^\top\in\R{3}$, and defining Lie algebra isomorphism $\cL\colon\R{6}\to\fse(3)$ by
%$$y=(y_1,\ldots,y_6)^\top\mapsto
%\begin{pmatrix}
%0 & -y_6 & y_5 & y_1 \\
%y_6 & 0 & -y_4 & y_2 \\
%-y_5 & y_4 & 0 & y_3 \\
%0 & 0 & 0 & 0
%\end{pmatrix},
%$$}
the equations of motion \eqref{014} can be cast in the form
\begin{equation}\label{000_1}
\begin{pmatrix}
\cL^{-1}(g^{-1}\dot g) \\
\dot s
\end{pmatrix}=V_1(s)u_1+V_2(s)u_2 \eqqcolon
\begin{pmatrix} 
\cL^{-1} \xi_1(s)\\ 
\hat\be_1^{\R{2}}
\end{pmatrix} u_1+
\begin{pmatrix}
\cL^{-1} \xi_2(s)\\ 
\hat\be_2^{\R{2}} 
\end{pmatrix} u_2.
\end{equation}
In \eqref{000_1}, we notice that $g^{-1}\dot g\in \fse(3)$; the action of $g^{-1}$ on an element $\dot g$ of the tangent space $T_g SE(3)$ can be written as 
$$\begin{pmatrix}
R^{-1}(\alpha,\beta,\gamma) & -\tau \\
0 & 1
\end{pmatrix} \dot g;$$
$V_1(s)$, $V_2(s)$, $\cL^{-1}\xi_1(s)$, $\cL^{-1}\xi_2(s)$ can be found in \eqref{021}, \eqref{015}, \eqref{016N}. %, and the $\hat\be_i^{\R{2}}$ are the basis vectors in $\R2$.
Finally, $u_1,u_2\colon[0,T]\to\R{}$ are the \emph{control functions}.  It is a well-known fact that if $u_1$, $u_2$ are taken in $L^\infty(0,T)$, there exists a unique absolutely continuous solution to \eqref{000_1} \cite[Lemma~2.1]{JS}.

We now remark that, since $\cL$ is an isomorphism, system \eqref{000_1} is exactly a control system on the Lie group $SE(3)$ according to Definition \ref{def:controlsystem}, and thus the control vector fields are equivariant with respect to the $SE(3)$ action, as pointed out in Remark~\ref{remark}(i).

\smallskip

\noindent \textit{Step 2.} By Remark~\ref{remark}(ii) and Theorem~\ref{ostro}(i), 
%Using the fact that given two vector fields that are equivariant with respect to a group action also their Lie bracket is\footnote{\blu{Is really needed? Maybe it suffices the orbit theorem}}, 
to prove the fiber controllability of the system at a point $(h,s^*)$ it suffices to compute the Lie brackets of the vector fields $V_i$ at a point $(e,s^*)$ and to show that they generate any directions in the Lie algebra $\fse(3)$. %and finally, since any point $(s,h)$ belongs to the orbit of $(s^*,e)$, the Orbit theorem states that the Lie algebra of the $V_i$s is fully generated everywhere. %\cite{BL2005,CMOZ2002}. 
A simple computation of these Lie brackets at the point $(e,s^*)=\big(e,(\varphi^*,\vartheta^*)\big)=\big(e,(\frac{\pi}{2},0)\big)$ yields
\begin{equation*}%\label{016}
V_3\coloneqq  [V_1,V_2]_{(e,s^*)}=
\left(
\begin{array}{c}
0\\
\displaystyle \frac{6 C_\tau  L \left(3 C\tau -2 C_\perp  L^2\right)}{\left(15 C_\tau +2 C_\perp  L^2\right)^2}-\frac{C_\perp  L}{4 (C_\perp +C_\parallel )}\\
0\\
\displaystyle-\frac{351 C_\tau ^2+4 C_\perp ^2 L^4+120 C_\tau  C_\perp  L^2}{2 \left(15 C_\tau +2 C_\perp  L^2\right)^2}\\
0\\
\displaystyle \frac{36 C_\tau  C_\perp  L^2}{\left(15 C_\tau +2 C_\perp  L^2\right)^2}\\
0\\
0
\end{array}
\right),
\end{equation*}
\begin{equation*}%\label{017}
V_4\coloneqq  [V_1,V_3]_{(e,s^*)}=
\left(
\begin{array}{c}
0\\
\displaystyle %-\frac{C_\perp }{C_\perp +C_\parallel }
\frac{6 C_\tau  L(567 C_\tau ^2-4 C_\perp ^2 L^4+132 C_\tau  C_\perp  L^2)}{(15 C_\tau +2 C_\perp 
   L^2)^3}-\frac{C_\perp C_\parallel L}{2(C_\perp +C_\parallel)^2}\\
   0\\
   \displaystyle \frac{9 C_\tau  \left(927 C_\tau ^2-4 C_\perp ^2 L^4+180 C_\tau  C_\perp L^2\right)}{\left(15 C_\tau +2 C_\perp  L^2\right)^3}\\
   0\\
   \displaystyle \frac{24 C_\tau  C_\perp  L^2 \left(21 C_\tau +10 C_\perp  L^2\right)}{\left(15 C_\tau +2 C_\perp  L^2\right)^3}\\
   0\\
   0
\end{array}
\right),
\end{equation*}
\begin{equation*}%\label{018}
V_5\coloneqq [V_2,V_3]_{(e,s^*)}=
\left(
\begin{array}{c}
\displaystyle \frac{L \left(9 C_\tau ^2 (17 C_\perp -8 C_\parallel )+4 C_\perp ^3 L^4+12 C_\tau  C_\perp  L^2 (9 C_\perp +4 C_\parallel )\right)}{4 (C_\perp +C_\parallel ) \left(15 C_\tau +2 C_\perp  L^2\right)^2}\\
0\\
0\\
0\\
\displaystyle -\frac{351 C_\tau^2+4 C_\perp ^2 L^4+120 C_\tau  C_\perp  L^2}{2 \left(15 C_\tau +2 C_\perp  L^2\right)^2}\\
0\\
0\\
0
\end{array}
\right),  
\end{equation*}
\begin{equation*}%\label{019}
V_6\coloneqq [V_1,V_5]_{(e,s^*)}=
\left(
\begin{array}{c}
\displaystyle \frac{C_\perp C_\parallel L}{2(C_\perp +C_\parallel)^2}-\frac{6 C_\tau L (567 C_\tau ^2-4 C_\perp ^2 L^4+132 C_\tau  C_\perp  L^2)}{(15 C_\tau +2 C_\perp  L^2)^3} \\
0\\
0\\
0\\
\displaystyle \frac{9 C_\tau  \left(927 C_\tau ^2-4 C_\perp ^2 L^4+180 C_\tau  C_\perp  L^2\right)}{\left(15 C_\tau +2 C_\perp  L^2\right)^3}\\
0\\
0\\
0
\end{array}
\right). 
\end{equation*}

Let $v_i$ $(i=1,\ldots,6)$ be the vectors $V_i$ without the last two components, so that $(v_1|\cdots|v_6)$ is a $6\times 6$ matrix.
The computation of its determinant at the point $(\frac{\pi}{2},0)$, gives
\begin{equation}\label{020}
\delta\coloneqq\det(v_1|\cdots|v_6)_{(\frac{\pi}{2},0)}=\frac{p(C_\parallel, C_\perp, C_\tau, L)}{q(C_\parallel, C_\perp, C_\tau, L)},%-\frac{L^3 C_\perp^3 C_\parallel ^3 \sin^3\varphi_t}{64(C_\perp+C_\parallel+(C_\parallel-C_\perp)\cos\varphi_t)^5}.
\end{equation}
where $p$ %= p(C_\parallel, C_\perp, C_\tau, L)$ 
and $q$ %= q(C_\parallel, C_\perp, C_\tau, L)$ 
are polynomials %in the variables $(C_\parallel,C_\perp,C_\tau,L)$
whose explicit expressions are
\begin{equation}\label{pq}
\begin{split}
p= & C_\perp ^2 L^5 \big[ 27 C_\tau ^2 (11 C_\perp ^2+32 C_\perp  C_\parallel -4 C_\parallel ^2)+6 C_\tau  C_\perp  L^2 (C_\perp ^2+25 C_\perp  C_\parallel +4 C_\parallel ^2) \\
&+4 C_\perp ^3 C_\parallel L^4 \big] \cdot  \big[8 C_\perp ^4 C_\parallel L^6  -81 C_\tau^3 (61 C_\perp ^2+160 C_\perp  C_\parallel +164 C_\parallel ^2) \\
& +12 C_\tau  C_\perp ^2 L^4 (C_\perp ^2+30 C_\perp  C_\parallel +4 C_\parallel ^2)+18 C_\tau ^2 C_\perp  L^2 (18 C_\perp ^2+85 C_\perp  C_\parallel -72 C_\parallel ^2)\big],\\
q= & 32 (C_\perp +C_\parallel )^5 (15 C_\tau +2 C_\perp  L^2)^6.
\end{split}
\end{equation}
Notice that $q$ never vanishes, wheres the set $\{(C_\parallel, C_\perp, C_\tau, L)\in (0,+\infty)^4: p(C_\parallel, C_\perp, C_\tau, L)=0\}$ has zero four-dimensional Lebesgue measure.
%which is possibly zero only for a finite number of choices of the parameters, \hs{HS: how do we know it can only be finitely many? I find it easier to believe that there are either zero or infinitely many parameter combinations that make $p=0$.} \red{MM: bleah} and
%\begin{equation}
%\label{q}
%q= 32 (C_\perp +C_\parallel )^5 (15 C_\tau +2 C_\perp  L^2)^6,
%\end{equation}
%which never vanishes.
%Inspecting the expression above, we find that $t\mapsto\delta_t$ is always different from zero provided that $\varphi_t\neq0,\pi$.
%which is always non-zero (recall that $\varphi_t\in(0,\pi)$ for all $t\geq0$, by the positions in \eqref{050}).
This proves that, through the iterated Lie brackets, it is possible to generate the $6$-dimensional Lie algebra $\fse(3)$ at the point $(h,s^*)=(e,(\frac{\pi}{2},0))$. 
Fiber controllability at $(h,s^*)$ follows.

\noindent \textit{Step 3.} Recalling that for any $s\in\cS$ there exists a point $h'\in G$ such that $(h',s)$ belongs to the orbit of the point $(h,s^*)$ (since the shape variable can be steered directly by means of the control functions  $u_i$, invoking that the group action is free), and that the vector fields $V_i$ are analytic\footnote{$V_1$ and $V_2$ are analytic from \eqref{015}, \eqref{016N} and thus also their Lie brackets. }, the Orbit Theorem~\ref{orbit} states that the Lie algebra generated by the vector fields $V_1$ and $V_2$ has the same dimension at any point along the orbit. 
Finally, thanks to the equivariance of the vector fields with respect to the group action (see Remark~\ref{remark}(ii)), it is easy to see that $V_1$ and $V_2$ also generate the Lie algebra $\fse(3)$ at any points of the form $(e,s)$.
Fiber controllability follows from Theorem \ref{ostro}(ii).
\end{proof}
\begin{remark}\label{important}
By standard results on control theory \cite{Coron,trelat}, controllability is ensured with controls in $L^\infty$, thus for any final time $T<+\infty$ the $2$-link swimmer is fiber controllable by means of absolutely continuous shape parameters $(\vartheta_t,\varphi_t)\in\R{}\times\T$ for all $t\in[0,T]$ (see the $\cS$-component of \eqref{000_1}).
%This justifies, \emph{a posteriori}, that the determinant $\delta_t$ in \eqref{020} is different from zero for almost every $t\in[0,T]$.
\end{remark}

\begin{proposition}
\label{scallop_theorem}
If $C_\tau=0$ the $2$-link swimmer is not fiber controllable, and we recover the well-known scallop theorem.
\end{proposition}
\begin{proof}
Let us consider the the basis vectors $\hat\be_1^{\R{6}},\ldots,\hat\be_6^{\R{6}}$, whose image through $\cL$ is a basis of the Lie algebra $\fse(3)$.
Setting $C_\tau=0$ in \eqref{015} and \eqref{016N} we have that 
\begin{equation}\label{015N}
V_1=\left(
\begin{array}{c}
\displaystyle \frac{LC_\perp\cos\vartheta_t\sin^2\frac{\varphi_t}{2}}{2(C_\perp+C_\parallel+(C_\parallel-C_\perp)\cos\varphi_t)}\\
\displaystyle \frac{LC_\perp\sin\vartheta_t\sin^2\frac{\varphi_t}{2}}{2(C_\perp+C_\parallel+(C_\parallel-C_\perp)\cos\varphi_t)}\\
\displaystyle \frac{LC_\perp\sin\varphi_t}{4(C_\perp+C_\parallel+(C_\parallel-C_\perp)\cos\varphi_t)}\\
\displaystyle \frac{\sin\vartheta_t}{2}\\
\displaystyle -\frac{\cos\vartheta_t}{2}\\
0\\
1\\
0
\end{array}
\right)\quad\text{and}\quad
V_2=\left(
\begin{array}{c}
0\\
0\\
0\\
0\\
0\\
-1\\
0\\
1
\end{array}
\right).
\end{equation}
The first six components of the $V_i$'s belong to the Lie algebra $\fse(3)$ via the isomorphism $\cL$, so that we will work with the $v_i$'s defined in the proof of Theorem~\ref{th:contr}.
The expression of $V_2$ in \eqref{015N} yields that
%\begin{equation}\label{e6}
$\hat\be_6^{\R{6}}=-v_2$.
%\end{equation}
Because of this, we do not have two real shape parameters, because $-\dot\vartheta$ coincides with one direction of the Lie algebra. As a result, whenever we move the angle theta, the system reacts with a counter-rotation by the same angle, so that the $2$-link swimmer does not leave the plane determined by the initial angle. 
In this case, the angle $\vartheta$ cannot be considered as a proper shape parameter. 
Therefore, for $C_\tau=0$, the system has only one shape parameter which makes it equivalent to a planar scallop subject to the well-known scallop theorem (see \cite{Purcell1977}).
\end{proof}

\section{The $N$-link swimmer}\label{sec:ext}
In this section, we extend the results obtained in the previous sections to the 
%In Subection~\ref{sec:Nlink}, we treat the general 
$N$-link swimmer. %, with $N>2$. %; in Subsection~\ref{sec:root}, we compare our results with the controllability theorem obtained in \cite{root} for a three-dimensional $3$-link swimmer.
%\red{extend/write more? Henry? you're good with English.}
%\subsection{The $N$-link swimmer}\label{sec:Nlink}
We consider a slender swimmer composed of a chain of $N>2$ links of length $\ell_i\geq0$ hinged at their extremities and moving in an infinite viscous fluid.
In order to avoid degeneracy, we require that there exist at least  $i,j\in\{1,\ldots,N\}$, $i\neq j$, such that $\ell_i>0$ and $\ell_j>0$.

To provide a dynamical description of the $N$-link swimmer, we follow the construction of Section~\ref{sec:dyn}: each link is described by two angles $\vartheta^{(i)}\in\R{}$, $\varphi^{(i)} \in \T$ that identify the direction of the link with respect to the co-moving frame.
%As in Section \ref{sec:dyn}, t
The angles $\{\vartheta^{(i)}, \varphi^{(i)}\}_{i=2}^N$ are the shape parameters of the system and we will prove that the swimmer is able to move in the fluid once the time evolution of the $2N-2$ functions $t \mapsto \vartheta_t^{(i)}$ and $t \mapsto \varphi_t^{(i)}$ are given.

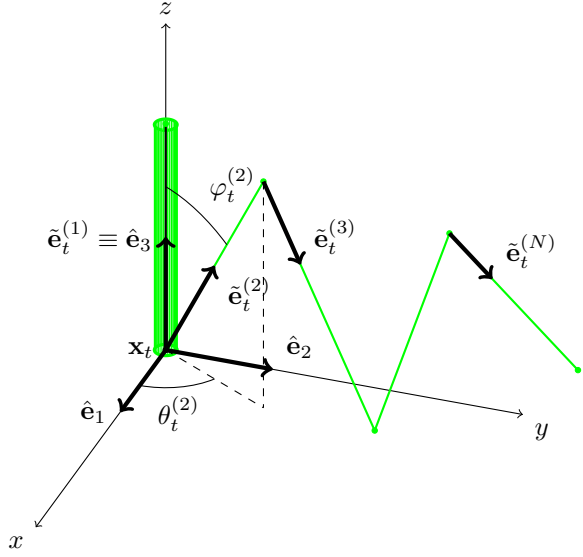
\begin{figure}[t]
	\centering
	\begin{tikzpicture}[scale=5,tdplot_main_coords]
	\coordinate (O) at (0,0,0);
	\draw[thin,->] (0,0,0) -- (1,0,0) node[anchor=north east]{$x$};
	\draw[thin,->] (0,0,0) -- (0,1,0) node[anchor=north west]{$y$};
	\draw[thin,->] (0,0,0) -- (0,0,1) node[anchor=south]{$z$};
	\tdplotsetcoord{P}{\rvec}{\thetavec}{\phivec}
	\tdplotsetcoord{P/2}{\rvec /2}{\thetavec}{\phivec}
	\foreach \t in {10,20,...,360}% generatrices
		\draw[green] ({0.03*cos(\t)},{0.03*sin(\t)},0)
      --({0.03*cos(\t)},{0.03*sin(\t)},{0.69});
	\draw[green,very thick] (0,0,0) % lower circle
		\foreach \t in {5,10,...,360}
			{--({0.03*cos(\t)},{0.03*sin(\t)},0)};
	\draw[green,very thick] (0,0,{0.69}) % upper circle
		\foreach \t in {10,20,...,360}
			{--({0.03*cos(\t)},{0.03*sin(\t)},{0.69})};
	\coordinate (P1) at (0.6899, 0.8363, 0.2928);
	\coordinate (P2) at (0.8430, 1.1014, 1.0319);
	\coordinate (P3) at (1.4087, 1.6671, 1.0319);
	\coordinate (e2) at (0.3633, 0.5097, 0.5595);
	\coordinate (eN) at (1.0315, 1.2900, 1.0319);
	\draw[green,thick] (O) -- (P);
	\draw[thick] (O) -- (Pz);
	\draw[dashed] (O) -- (Pxy);
	\draw[dashed] (P) -- (Pxy);
	\draw [ fill ] (O) circle [ radius = 0.2 pt];
	\draw [green, fill ] (P) circle [ radius = 0.2 pt];
	\draw [ green, fill ] (Pz) circle [ radius = 0.2 pt];
	\draw [green, fill ] (P1) circle [ radius = 0.2 pt];
	\draw [green, fill ] (P2) circle [ radius = 0.2 pt];
	\draw [green, fill ] (P3) circle [ radius = 0.2 pt];
	\draw [ultra thick, ->] (O) -- (P/2) node[below right] {$\tilde{\mathbf{e}}_t^{(2)}$};
	\draw [ultra thick, ->] (O) -- (P/2z) node[left] {$\tilde{\mathbf{e}}_t^{(1)} \equiv \hat{\mathbf{e}}_3$};
	\draw [ultra thick, ->] (O) -- (0.35,0,0) node[left] {$\hat{\mathbf{e}}_1$};
	\draw [ultra thick, ->] (O) -- (0,0.3,0) node[above right] {$\hat{\mathbf{e}}_2$};
	\draw [thin] (O) -- (O) node[left] {$\mathbf{x}_t$};
	\draw [green,thick] (P) -- (P1);
	\draw [green,thick] (P1) -- (P2);	
	\draw [green,thick] (P2) -- (P3);	
	\draw [ultra thick, ->] (P) -- (e2) node[above right] {$\tilde{\mathbf{e}}_t^{(3)}$};
	\draw [ultra thick, ->] (P2) -- (eN) node[above right] {$\tilde{\mathbf{e}}_t^{(N)}$};
	\tdplotdrawarc{(O)}{0.2}{0}{\phivec}{anchor=north}{$\theta_t^{(2)}$}
	\tdplotsetthetaplanecoords{\phivec}
	\tdplotdrawarc[tdplot_rotated_coords]{(0,0,0)}{0.5}{0}%
	{\thetavec}{anchor=south west}{$\varphi_t^{(2)}$}
	\end{tikzpicture}
	\caption{Co-moving frame of the $N$-link swimmer.}
	\label{N-link-figure}
\end{figure}

The unit vectors that describe the directions of the links are (see Figure~\ref{N-link-figure})
\begin{equation*}
	\tilde{\mathbf{e}}_{t}^{(1)} \coloneqq \hat{\be}_3 \, , \qquad 
	\tilde{\mathbf{e}}_{t}^{(i)} \coloneqq
	\begin{pmatrix} 
	\cos \vartheta_t^{(i)} \sin \varphi_t^{(i)}\\
	\sin \vartheta_t^{(i)} \sin \varphi_t^{(i)} \\
	\cos \varphi_t^{(i)}
	\end{pmatrix} \, , \quad 
	i \in \{2, \dots, N\},
\end{equation*}
while, in the lab frame, the positions along the links, each of which is parametrized by an arc-length coordinate $s \in [0, \ell_i]$, $i\in\{1,\ldots,N\}$, are
\begin{equation}\label{4.2}
	\bx_t^{(1)} (s) =  \bx_t + R_t s \tilde{\be}_t^{(1)} \, , \qquad
	\mathbf{x}_t^{(i)} (s) = \mathbf{x}_t + R_t \bigg[ \sum_{j=2}^{i-1} \ell_j \tilde{\mathbf{e}}_{t}^{(j)} + s \tilde{\mathbf{e}}_{t}^{(i)} \bigg] \, , \quad i \in \{2,\dots,N\},
\end{equation}
where $t\mapsto\bx_t$ is the position of the joint between link $1$ and link $2$ with respect to the origin of the lab frame and $t\mapsto R_t$ is its orientation.
By \eqref{4.2} and Resistive Force Theory, we can compute the densities of viscous force and torque $\bef_t^{(i)}(s)$ and $\btau_t^{(i)}(s)$ as in \eqref{003}. 
To derive the equations of motion of the swimmer, both the entries of the grand resistance matrix $\tilde \cM_t$ and the viscous force and torque  $\tilde{\mathbf{F}}_t^{\mathrm{sh}}$ and $\tilde{\mathbf{T}}_t^{\mathrm{sh}}$ due to the shape change must be computed.
The block entries of the grand resistance matrix are
\begin{equation}\label{43}
	\tilde{K}_t \coloneqq \sum_{i=1}^{N} \tilde{K}_t^{(i)} \, , \qquad
	\tilde{C}_t \coloneqq %\tilde{C}_t^{(1)} + 
	\sum_{i=1}^N \tilde{C}_t^{(i)} \, , \qquad
	\tilde{J}_t \coloneqq %\tilde{J}_t^{(1)} + 
	\sum_{i=1}^N \tilde{J}_t^{(i)}
\end{equation}
where $\tilde{K}_t^{(i)}$, $\tilde{C}_t^{(i)}$ for $i=1,\ldots,N$, are given by% are the same as in \eqref{009} %(NOTE: also $\tilde{K}_t^{(i)}$ are the same of \eqref{009}, do they have to be re-written?), 
%and, for $i=2,\ldots,N$,
\begin{equation}\label{44}
\begin{split}
	\tilde{K}_t^{(i)}  \coloneqq & [(C_\parallel-C_\perp)\tilde\be^{(i)}_t\otimes\tilde\be^{(i)}_t+C_\perp I]\ell_i \,, \\
	\tilde{C}_t^{(i)}  \coloneqq & [(C_\parallel-C_\perp)\tilde\be^{(i)}_t\otimes\tilde\be^{(i)}_t+C_\perp I]\ell_i \bigg(\sum_{j=2}^{i-1} \ell_j \tilde{E}_t^{(j)} \bigg) + \frac{\ell_i^2}{2} C_\perp \tilde{E}^{(i)}_t  \,, 
	\end{split}
	\end{equation}
and 
\begin{equation*}
	\begin{split}
	\tilde{J}_t^{(1)}  \coloneqq &  \frac{\ell_1^3}{3} C_\perp \big[ I - \tilde{\be}_t^{(1)} \otimes \tilde{\be}_t^{(1)}\big] +\ell_1C_\tau \tilde{\be}_t^{(1)} \otimes \tilde{\be}_t^{(1)},\\
	\tilde{J}_t^{(i)}  \coloneqq & - \ell_i (C_\parallel - C_\perp) \bigg(\sum_{j=2}^{i-1} \ell_j \tilde{E}_t^{(j)} \bigg) ( \tilde{\be}_t^{(i)} \otimes \tilde{\be}_t^{(i)} ) \bigg(\sum_{j=2}^{i-1} \ell_j \tilde{E}_t^{(j)} \bigg) \\
	& - \ell_i C_\perp \bigg(\sum_{j=2}^{i-1} \ell_j \tilde{E}_t^{(j)} \bigg)^2 - \frac{\ell_i^2}{2} C_\perp \tilde{E}_t^{(i)} \bigg(\sum_{j=2}^{i-1} \ell_j \tilde{E}_t^{(j)} \bigg)  \\
	& - \frac{\ell_i^2}{2} C_\perp \bigg(\sum_{j=2}^{i-1} \ell_j \tilde{E}_t^{(j)} \bigg) \tilde{E}_t^{(i)}  + \frac{\ell_i^3}{3} C_\perp \big[ I - \tilde{\be}_t^{(i)} \otimes \tilde{\be}_t^{(i)}\big] \,,\quad\text{for}\quad i=2,\ldots N.
\end{split}
\end{equation*}
In the expressions above, the matrices $\tilde E_t^{(i)}$ represent the vector product $\tilde\be_t^{(i)}\times$, as in \eqref{009}.
The expression of the grand resistance matrix $\tilde \cM_t$ given in \eqref{013} still holds, using the formulas for the blocks in \eqref{43} and \eqref{44}.   %for every $i \in \{ 2, \dots, N\}$. 
The vectors $\tilde{\mathbf{F}}_t^{\mathrm{sh}}$ and $\tilde{\mathbf{T}}_t^{\mathrm{sh}}$ are 
\begin{equation}\label{45}
\begin{split}
	\tilde{\mathbf{F}}_t^{\mathrm{sh}}  \coloneqq & \sum_{i=2}^{N} \bigg[ [(C_\parallel-C_\perp)\tilde\be^{(i)}_t\otimes\tilde\be^{(i)}_t+C_\perp I]\ell_i \bigg(\sum_{j=2}^{i-1} \ell_j \dot{\tilde{\be}}_t^{(j)} \bigg) + \frac{\ell_i^2}{2} C_\perp \dot{\tilde{\be}}_t^{(i)}\bigg] \\
	\tilde{\mathbf{T}}_t^{\mathrm{sh}}  \coloneqq & \sum_{i=2}^N \bigg[ \ell_i \bigg(\sum_{j=2}^{i-1} \ell_j \tilde{E}_t^{(j)} \bigg)[(C_\parallel-C_\perp)\tilde\be^{(i)}_t\otimes\tilde\be^{(i)}_t+C_\perp I] \bigg(\sum_{j=2}^{i-1} \ell_j \dot{\tilde{\be}}_t^{(j)} \bigg)  \\
	& + \frac{\ell_i^2}{2} C_\perp \tilde{E}_t^{(i)} \bigg(\sum_{j=2}^{i-1} \ell_j \dot{\tilde{\be}}_t^{(j)} \bigg) + \frac{\ell_i^2}{2} C_\perp \bigg(\sum_{j=2}^{i-1} \ell_j \tilde{E}_t^{(j)} \bigg) \dot{\tilde{\be}}_t^{(i)} + \frac{\ell_i^3}{3} C_\perp \tilde{E}_t^{(i)} \dot{\tilde{\be}}_t^{(i)} \bigg] \, ,
\end{split}
\end{equation}
%(NOTE: the matrices and the vectors could be written in a more compact form taking into account that  $\tilde{K}_t^{(i)}$ appears several times and $\dot{\tilde{\be}}_t^{(1)} = 0$ once we fix the first link as done in Section \ref{sec:dyn})
so that, analogously to \eqref{012}, the equations of motion read
\begin{equation}\label{N012}
\mathbf{0}=
\vec{\bF_t}{\bT_t}=\tilde \cM_t\mat{R_t^{-1}}{0}{0}{R_t^{-1}}\vec{\dot\bx_t}{\bomega_t}+\vec{\tilde\bF^\sh_t}{\tilde\bT^\sh_t}.
\end{equation}
By recalling that $\tilde \cM_t$ is positive definite, and therefore invertible, \eqref{N012} can be written as
\begin{equation}\label{N001}
\begin{split}
\begin{pmatrix}
R_t^{-1}\dot\bx_t\\
R_t^{-1}\bomega_t\\
\dot\varphi_t^{(2)}\\
\dot\vartheta_t^{(2)}\\
\vdots \\
\dot\varphi_t^{(N)}\\
\dot\vartheta_t^{(N)}
\end{pmatrix}=\sum_{i=2}^N \Big[ & V_1^{(i)}\big(\{(\vartheta_t^{(j)},\varphi_t^{(j)})\}_{j=2}^N\big)u_1^{(i)} +V_2^{(i)}\big(\{(\vartheta_t^{(j)},\varphi_t^{(j)})\}_{j=2}^N\big)u_2^{(i)}\big) \Big ],
\end{split}
\end{equation}
where, for $i=2,\ldots,N$,  $V_1^{(i)}$ and $V_2^{(i)}$ are vector fields with $6+2(N-1)=2N+4$ components.
The following theorem, whose proof is a byproduct of the controllability Theorem~\ref{th:contrN}, holds.
\begin{theorem}\label{th:exuniN}
Let $(\bar\bx,\bar R)\in\R{3}\times SO(3)$ be given.
There exists a unique absolutely continuous solution $(\bx_t,R_t)\colon[0,+\infty)\to \R{3}\times SO(3)$ to the Cauchy problem for \eqref{N001} with initial condition $(\bx_0,R_0)=(\bar\bx,\bar R)$, for any controls $u_1^{(i)},u_2^{(i)}\in L^\infty(0,+\infty)$ for $i=2,\ldots,N$.
\end{theorem}

\begin{theorem}[Controllability of the $N$-link]\label{th:contrN}
The $N$-link swimmer is fiber controllable in the sense of Definition~\ref{def:contr} %(i) 
for almost every lengths $\ell_i$ ($i=1,\ldots,N$) of the links.
%A basis of the Lie algebra $\fse(3)$ is explicitly constructed.
\end{theorem}
\begin{proof}
The proof follows the reasoning of that of \cite[Theorem~3.1]{GMZ}, where it is proved that the controllability of a planar $N$-link swimmer follows from that of a planar Purcell $3$-link swimmer.
In the present case, from the controllability of the $2$-link swimmer in three dimensions, together with the analyticity of the vector fields $\{V_1^{(i)},V_2^{(i)}\}_{i=2}^N$ (introduced in \eqref{N001}) with respect to the $\ell_i$'s, we will be able to deduce the controllability of the $N$-link swimmer.

More precisely, by setting $\ell_1=\ell_2\eqqcolon L$ and  $\ell_i=0$ for all $i=3,\ldots,N$, we reduce the $N$-link swimmer to a $2$-link swimmer, which can be described as in Section~\ref{sec:dyn}. 
In particular, the equations of motion \eqref{N001} read
\begin{equation*}%\label{N002}
\begin{split}
\begin{pmatrix}
R_t^{-1}\dot\bx_t\\
R_t^{-1}\bomega_t\\
\dot\varphi_t^{(2)}\\
\dot\vartheta_t^{(2)}\\
\vdots \\
\dot\varphi_t^{(N)}\\
\dot\vartheta_t^{(N)}
\end{pmatrix}=  & W_1^{(2)}(\vartheta_t^{(2)},\varphi_t^{(2)})u_1^{(2)} +W_2^{(2)}(\vartheta_t^{(2)},\varphi_t^{(2)})u_2^{(2)},
\end{split}
\end{equation*}
where the first eight components of $W_1^{(2)}$ and $W_2^{(2)}$ are obtained from those of $V_1^{(2)}$ and $V_2^{(2)}$, respectively, and the last $2N-4$ components of both $W_1^{(2)}$ and $W_2^{(2)}$ are zero.
Clearly, the first eight components of $W_1^{(2)}$ and $W_2^{(2)}$ are precisely the $V_1$ and $V_2$ in \eqref{015}.

By Theorem~\ref{th:contr}, the vector fields $W_1^{(2)}$ and $W_2^{(2)}$ generate all of the Lie algebra $\fse(3)$.
Indeed, by taking the iterated Lie brackets of $W_1^{(2)}$ and $W_2^{(2)}$ evaluated at $(e,s^*)=(e,(\frac{\pi}{2},0))$ as we did in the proof of Theorem~\ref{th:contr}, and by constructing the corresponding $w_1^{(2)},\ldots,w_6^{(2)}$, formula \eqref{020} holds:
\begin{equation}\label{N003}
\delta=\det\big(w_1^{(2)}|\cdots|w_6^{(2)}\big)_{(\frac{\pi}{2},0)}=\frac{p(C_\parallel,C_\perp, C_\tau,L)}{q(C_\parallel,C_\perp, C_\tau,L)} 
\end{equation}
with the same $p$ and $q$ defined in \eqref{pq}, and again it does not vanish for almost any choice of modeling parameters $(C_\parallel,C_\perp, C_\tau,L)$. 
Therefore, the vector fields $W_1^{(2)}$ and $W_2^{(2)}$ generate the $6$-dimensional Lie algebra $\fse(3)$ at the point $(e,s^*)$. As done for the $2$-link swimmer, we argue that from the analyticity of the vector fields and from the Orbit Theorem~\ref{orbit}, they generate the Lie algebra $\fse(3)$ at any point $(e,s)$.
Thus fiber controllability at any points $(h,s)$ follows for a swimmer with links of lengths $\ell_1=\ell_2=L$ and $\ell_i=0$ for $i>2$.  %provided $\varphi_t^{(2)}\notin\{0,\pi\}$ for almost every $t\in[0,T]$.
Taking \eqref{44} and \eqref{45} into account, it is easy to observe that the vector fields $\{V_1^{(i)},V_2^{(i)}\}_{i=2}^N$ in \eqref{N001} depend analytically on $\ell_1,\ldots,\ell_N$, so that 
\begin{equation}\label{N006}
\big(\ell_1,\ldots,\ell_N\big)\mapsto\delta=\det\big(\text{Lie brackets of $v_1^{(2)},v_2^{(2)}$}\big)_{(\frac{\pi}{2},0)}
\end{equation}
also does.
In particular, \eqref{N003} is the map in \eqref{N006} evaluated at 
\begin{equation}\label{N005}
\big(L,L,0,\ldots,0\big).
\end{equation}
Since \eqref{N003} is different from zero, the analytic map in \eqref{N006} will stay away from zero %in a neighbourhood of the point in \eqref{N005}.
%Therefore, 
for almost every lengths $\ell_i$'s of the links; fiber controllability is proved. %\hs{HS: Is this right? It looks like we just proved fiber controllability in a neighbourhood of one point. That's not the same as almost everywhere.} %and any shape parameters $\{\vartheta_t^{(i)},\varphi_t^{(i)}\}_{i=2}^N$ 
%controllability is ensured.
\end{proof}

\section{Optimal control problems}\label{sec:optimality}
In this section we tackle some optimality problems for the $2$-link swimmer whose solution we can characterise.
The generalisation to the $N$-link swimmer are easily deduced by consideration of some geometric constraints, such as non interpenetration. %\hs{HS: Is this true? It doesn't seem obvious that geometric constraints will preserve existence of optima}
%\red{straightforward}\footnote{Henry: is there any need to comment it? Basically, one should extend the results to the $N$-link enforcing the constraint of no interpenetration. I think it is an open constraint (some quantities need to be different from zero), so if one starts from a \emph{good} configuration maybe it can keep staying, with/by continuity, in a(nother) good configuration for a little while. We think it's too complicated to specify all the conditions in formulae here, so we should just find a nice/swift way to convey the message.} and can be 
%We will present our results for the $2$-link swimmer because it allows for a more detailed analysis.
Recalling the notation of Section~\ref{sec:prel}, given $(g,s)\in G\times\cS$ the status variable, and $u\in U$, where $U\subset \R{n}$ is the compact set of controls, solving a generic control problem for \eqref{000} amounts to minimising the time integral of a Lagrangian $\scrL\colon G\times\cS\times U\to\R{+}$ under suitable constraints, namely
\begin{equation}\label{OCP}%\tag{OCP}
\begin{cases}
\displaystyle \inf \bigg\{\int_0^{t_f} \scrL(g(t),s(t),u(t))\,\de t \bigg\}, \\
\text{$(g(t),s(t),u(t))\in G\times\cS\times U$ for every $t\in[0,t_f]$,} \\
\text{\eqref{000} holds for every $t\in[0,t_f]$,} \\
%(u_1(t),u_2(t))\in [-1,1]^2,\;\text{for every $t\in[0,t_f]$,} \\
g(0)=g_0, g(t_f)=g_1,
\end{cases}
\end{equation}
where $t_f>0$ is a final time, and $g_0$ and $g_1$ are prescribed initial and final status of the system, respectively.

Recalling \eqref{050}, for the $2$-link swimmer we have $G=SE(3)$ and $\cS=\R{}\times \T$. %; in particular, the condition that $s(t)\in\cS$ for every $t\in[0,t_f]$ in \eqref{OCP} is ensured by Remark~\ref{important}.
Finally, $u(t)=(u_1(t),u_2(t))\colon[0,t_f]\to U\subset\R{2}$, with $u_1$ and $u_2$ introduced in \eqref{014}.
Therefore, we can recast the optimal control problem \eqref{OCP} for the $2$-link swimmer as 
\begin{equation}\label{OCP2}%\tag{OCP}
\begin{cases}
\displaystyle \inf \bigg\{\int_0^{t_f} \scrL(g(t),\vartheta(t),\varphi(t),u_1(t),u_2(t))\,\de t \bigg\}, \\
\text{$(u_1(t),u_2(t))\in U$ for every $t\in[0,t_f]$,} \\
\text{\eqref{000_1} holds for every $t\in[0,t_f]$,} \\
%(u_1(t),u_2(t))\in [-1,1]^2,\;\text{for every $t\in[0,t_f]$,} \\
g(0)=g_0, g(t_f)=g_1.
\end{cases}
\end{equation}
We can state a general result for \eqref{OCP2}.
\begin{theorem}\label{thm:opt}
Let $\scrL\colon SE(3)\times\big(\R{}\times\T\big)\times U\to\R{+}$ be smooth.
Then there exists a %unique 
solution to the optimal control problem \eqref{OCP2}, namely, there exist an %unique 
absolutely continuous trajectory $\bar g\colon[0,t_f]\to SE(3)$, %unique 
absolutely continuous shape changes $(\bar\vartheta,\bar\varphi)\colon[0,t_f]\to \R{2}$, and %unique
bounded controls $(\bar u_1,\bar u_2)\colon[0,t_f]\to U$ such that 
\begin{equation*}%\label{opt1}
 \inf \bigg\{\int_0^{t_f} \scrL(g(t),\vartheta(t),\varphi(t),u_1(t),u_2(t))\,\de t \bigg\}=\int_0^{t_f} \scrL(\bar g(t),\bar\vartheta(t),\bar\varphi(t),\bar u_1(t),\bar u_2(t))\,\de t,
\end{equation*}
$(\bar u_1(t),\bar u_2(t))\in U$ for every $t\in[0,t_f]$, \eqref{000_1} holds for every $t\in[0,t_f]$, and $\bar g(0)=g_0$, $\bar g(t_f)=g_1$, where $g_0,g_1\in SE(3)$ are given.
\end{theorem}
\begin{proof}
By Theorem~\ref{th:contr}, the system is fiber controllable with bounded controls. 
It suffices to apply the ideas of Filippov Theorem, see \cite{agrachev}.
\end{proof}
We now discuss the solution to some specific optimal control problems.

%\red{Statement of the Optimal Control Problem for the minimal time and the linear quadratic cost.}
%\subsection{Minimal time optimal control problem}\label{sec:mtime}
The \emph{minimal time optimal control problem} for the $2$-link swimmer can be written as: given $g_0,g_1\in G$, solve
\begin{equation}\label{mtOCP}
\begin{cases}
\inf \big\{t_f: (u_1,u_2)\in[a,b]^2\big\}, \\
\text{\eqref{000_1} holds for every $t\in[0,t_f]$,} \\
(u_1(t),u_2(t))\in [a,b]^2,\;\text{for every $t\in[0,t_f]$,} \\
g(0)=g_0, g(t_f)=g_1.
\end{cases}
\end{equation}
\begin{theorem}\label{th:OCP}
For any $g_0,g_1\in SE(3)$, there exists a unique solution to \eqref{mtOCP}, namely there exist $\bar t_f\in\R{}$ and bounded controls $(\bar u_1,\bar u_2)\colon[0,\bar t_f]\to[a,b]^2$ of bang-bang type such that the infimum in \eqref{mtOCP} is attained at $\bar t_f$ for $(\bar u_1,\bar u_2)$.
%namely the infimum is indeed a minimum.
\end{theorem}
\begin{proof}
By taking $\scrL\equiv1$, Theorem~\ref{thm:opt} provides the existence of a solution to \eqref{mtOCP}.
Uniqueness follows from an adaptation of the proof of \cite[Theorem~15.3]{agrachev}.
A standard application of the Pontryagin Maximum Principle to \eqref{mtOCP} leads to obtaining that $(\bar u_1,\bar u_2)$ are of bang-bang type.
\end{proof}

We now turn to the \emph{optimal control for the power expended}.
Let us recall that, for a motion defined on the fixed time interval $[0,t_f]$, the power expended is defined as the scalar product of the force against the velocity, namely 
\begin{equation*}%\label{PE}
\cP\coloneqq \sum_{i=1}^N \int_0^{t_f} \int_0^{\ell_i} \Big[\scp{\bef_t^{(i)}(s)}{\dot \bx_t^{(i)}(s)}+\scp{\btau_t^{(i)}(s)}{\bomega_t}\Big]\de s\de t.
\end{equation*}
%\hs{In general, there is also a contribution to power expenditure due to rotation of the first link about its axis in the presence of torsional drag. For thin rods, this contribution can be neglected because the torsional drag coefficient $C_\tau$ varies with the square of the rod radius~\cite{Chwang1974} whereas $C_\parallel$ and $C_\perp$ vary inversely with the logarithm of the radius \cite{GH1955}.} 
Taking \eqref{004} and \eqref{005} into account, the power for the $2$-link swimmer analysed in previous Sections~\ref{sec:dyn} and~\ref{sec:controllability} ($N=2$ and $\ell_i=L$ for $i=1,2$) reads
\begin{equation*}%\label{PE2}
\begin{split}
\cP= & \int_0^{t_f} \int_0^{L} \Big[\scp{\bef_t^{(1)}(s)}{\dot \bx_t^{(1)}(s)}+\scp{\bef_t^{(2)}(s)}{\dot \bx_t^{(2)}(s)}+\scp{\btau_t^{(1)}(s)+\btau_t^{(2)}(s)}{\bomega_t}\Big] \de s\de t \\
= & \int_0^{t_f} \bigg[\frac{L^3 C_\perp \big(4C_\parallel + C_\perp +(4C_\parallel-C_\perp)\cos\varphi_t\big)}{24\big(C_\parallel+C_\perp+(C_\parallel-C_\perp) \cos\varphi_t\big)} \dot\varphi_t^2\\
&\phantom{\int_0^{t_f}\bigg[ } +\frac{12 C_\tau ^2 C_\perp  L^3 \sin ^2\varphi_t (5 (\cos (2 \varphi_t )+3)-12 \cos\varphi_t)}{\left(-36 C_\tau  \cos \varphi_t+45 C_\tau +\cos (2 \varphi_t ) \left(15 C_\tau -2 C_\perp  L^2\right)+2 C_\perp 
   L^2\right)^2}\dot\vartheta^2\bigg]\,\de t.
\end{split}
\end{equation*}
The power expended $\cP$ is expected to be a function of the shape parameters and of their velocities and in particular it is quadratic in the velocities.
%The explicit expression in \eqref{PE2} shows that for the $2$-link swimmer the power expended only depends on $\varphi_t$ and $\dot\varphi_t$.

Recalling that we posed $u_1=\dot\varphi$ and $u_2=\dot\vartheta$ (see \eqref{015} and \eqref{016N}), the optimal control problem for the power expended can be cast in the form \eqref{OCP2} by taking 
$$
\begin{aligned}\scrL(g,\vartheta,\varphi,u_1,u_2)=&\scrL_\cP(\varphi,u_1,u_2)\coloneqq  \frac{L^3 C_\perp \big(4C_\parallel + C_\perp +(4C_\parallel-C_\perp)\cos\varphi\big)}{24\big(C_\parallel+C_\perp+(C_\parallel-C_\perp) \cos\varphi\big)} u_1^2\\
&+\frac{12 C_\tau ^2 C_\perp  L^3 \sin ^2\varphi (5 (\cos (2 \varphi )+3)-12 \cos (\varphi ))}{\left(-36 C_\tau  \cos\varphi+45 C_\tau +\cos (2 \varphi ) \left(15 C_\tau -2 C_\perp  L^2\right)+2 C_\perp  L^2\right)^2}u_2^2\,,
\end{aligned}
$$
namely
\begin{equation}\label{OCPPE}%\tag{OCP}
\begin{cases}
\displaystyle \inf \bigg\{\int_0^{t_f} \scrL_\cP(\varphi(t),u_1(t),u_2(t))\,\de t \bigg\}, \\
\text{$(u_1(t),u_2(t))\in [a,b]^2$ for every $t\in[0,t_f]$,} \\
\text{\eqref{000_1} holds for every $t\in[0,t_f]$,} \\
%(u_1(t),u_2(t))\in [-1,1]^2,\;\text{for every $t\in[0,t_f]$,} \\
g(0)=g_0, g(t_f)=g_1.
\end{cases}
\end{equation}
\begin{theorem}\label{thm:PE}
Given $t_f>0$, for any $g_0,g_1\in SE(3)$, there exists a unique solution to \eqref{OCPPE}, namely there exist an %unique 
absolutely continuous trajectory $\bar g\colon[0,t_f]\to SE(3)$, %unique 
absolutely continuous shape changes $(\bar\vartheta,\bar\varphi)\colon[0,t_f]\to \R{}\times \T$, and %unique
bounded controls $\bar u_1\colon[0,t_f]\to [a,b]$,  $\bar u_2\colon[0,t_f]\to [a,b]$, either continuous or of bang-bang type%, and $\bar u_2\colon[0,t_f]\to [a,b]$ of bang-bang type 
such that 
\begin{equation*}%\label{opt1}
 \inf \bigg\{\int_0^{t_f} \scrL_\cP(\varphi(t),u_1(t),u_2(t))\,\de t \bigg\}=\int_0^{t_f} \scrL_\cP(\bar\varphi(t),\bar u_1(t),\bar u_2(t))\,\de t,
\end{equation*}
$(\bar u_1(t),\bar u_2(t))\in [a,b]^2$ for every $t\in[0,t_f]$, \eqref{000_1} holds for every $t\in[0,t_f]$, and $\bar g(0)=g_0$, $\bar g(t_f)=g_1$.
\end{theorem}
\begin{proof}
Theorem~\ref{thm:opt} provides the existence of a solution to \eqref{OCPPE}.
Uniqueness of %$\bar u_2$ follows from an adaptation of the proof of \cite[Theorem~15.3]{agrachev}, whereas uniqueness of 
$\bar u_1$ is implied by the strict convexity of $\scrL_\cP$ with respect to $u_1$ and $u_2$.
The regularity of $\bar u_1$ and $\bar u_2$ is a consequence of a standard application of the Pontryagin Maximum Principle\footnote{By stationarizing the Hamiltonian of the Pontryagin Maximum Principle with respect to $u_1$ and $u_2$, we obtain the control $\bar u_1$ and $\bar{u}_2$ together with their regularity: the functions $t\mapsto \bar u_1(t)$, $t\mapsto \bar u_2(t)$ are continuous if the stationary point belongs to $[a,b]^2$ for all $t\in[0,t_f]$; otherwise, it is of bang-bang type.}.
\end{proof}

\section{Conclusions and outlook}\label{sec:conclusions}

In this paper we studied the dynamics, controllability and optimal control problems for a $2$-link swimmer capable of performing fully three-dimensional shape changes.
In Section~\ref{sec:dyn}, we described the configuration and shape of the swimmer and derived the equations of motion of the $2$-link swimmer in a low Reynolds number flow by means of Resistive Force Theory and enforcing the so-called \emph{self-propulsion constraint} (setting the viscous force and torque equal to zero, see \eqref{012} and \eqref{014}).
Theorem~\ref{th:exuni} states the existence and uniqueness of the solution to the equations of motion \eqref{014}.
It is derived directly from Theorem~\ref{th:contr}, which is the main result of the paper and the core of Section~\ref{sec:controllability}.
The proof of Theorem~\ref{th:contr} is achieved by applying techniques from Geometric Control Theory.

In Section~\ref{sec:ext} we extended the results to the case of a general, fully three-dimensional $N$-link swimmer, exploiting the analyticity of the vector fields governing the dynamics.
Finally, in Section~\ref{sec:optimality}, we addressed two specific optimal control problems for the $2$-link swimmer, namely the minimal time optimal control problem and the minimisation of the power expended.
Both problems have an independent interest and find their relevance in the design of artificial micro-devices which mimic the motion of natural micro-organisms.

\smallskip

The results obtained in this paper focus on the self-propulsion case, as it is the first step towards the design of self-propelling micro-robots.
Nonetheless, it can be interesting for the applications, and object of future work, to extend the study to externally driven micro-swimmers.
This direction has already been pursued in the case of two-dimensional magneto-elastic swimmers: in \cite{MartaSORO,ADSGZ2017} a planar $N$-link is studied, showing that it can achieve a non-zero net displacement when actuated by a sinusoidal external magnetic field; in \cite{GP2015} local controllability of a $2$-link magneto-elastic swimmer is proved, wheres in \cite{Giraldi3d} the actuation of a three-dimensional $N$-link swimmer  by an external magnetic field is studied. Finally, we mention that the case of a multi-flagellar swimmer is studied in \cite{YRS}.
Imposing an external actuating field on the one hand has the benefit of \emph{helping} the swimmer to move and simplifying its design from the engineering point of view, while on the other hand makes the problem more challenging from the mathematical point of view.  
%since it requires to be able to track the shape changes to obtain the same controllability result.

\subsection*{Acknowledgments} 
RM developed part of this work (Sections~\ref{sec:dyn} and~\ref{sec:ext}) as part of his B.Sc.~thesis under the supervision of MM.
HS thanks the hospitality of the department of mathematics of the Politecnico di Torino and gratefully acknowledges partial support from the MIUR grant Dipartimenti di Eccellenza 2018-2022 (CUP: E11G18000350001) and support of the Natural Sciences and Engineering Research Council of Canada (NSERC), [funding reference number RGPIN-2018-04418].
Cette recherche a \'{e}t\'{e} financ\'{e}e par le Conseil de recherches en sciences naturelles et en g\'{e}nie du Canada (CRSNG), [num\'{e}ro de r\'{e}f\'{e}rence RGPIN-2018-04418].
MM is a member of the Gruppo Nazionale per l'Analisi Matematica, la Probabi\-lit\`a e le loro Applicazioni (GNAMPA) of the Istituto Nazionale di Alta Matematica (INdAM).
MZ is a member of the Gruppo Nazionale per la Fisica Matematica (GNFM) of the Istituto Nazionale di Alta Matematica (INdAM).
Both MM and MZ gratefully acknowledge support from the MIUR grant Dipartimenti di Eccellenza 2018-2022 (CUP: E11G18000350001).

\end{document}